\documentclass[reqno]{amsart}

\usepackage{amsfonts}
\usepackage{amssymb}
\usepackage{amsmath}
\usepackage{enumitem}
\usepackage{xcolor}
\usepackage{todonotes}

\setenumerate{label={\rm (\alph{*})}}
\usepackage{bbm,euscript,mathrsfs}
\usepackage{paper_diening}
\usepackage{graphicx}
\usepackage[top=1in, bottom=1.25in, left=1.10in, right=1.10in]{geometry}

\allowdisplaybreaks

\numberwithin{equation}{section}

\usepackage{tikz-cd}
\usetikzlibrary{lindenmayersystems}
\usetikzlibrary{decorations.pathreplacing}

\DeclareMathOperator{\Div}{div}

\newcommand{\R}{\mathbb R}
\newcommand{\N}{\mathbb N}
\newcommand{\dd}{\mathrm d}

\newcommand{\GG}{\boldsymbol{\mathfrak{G}}}

\renewcommand{\bfB}{B}

\newcommand{\dx}{\,\mathrm{d}x}
\newcommand{\dy}{\,\mathrm{d}y}

\newtheorem{theorem}{Theorem}[section]

\newtheorem{corollary}[theorem]{Corollary}
\newtheorem{remark}[theorem]{Remark}

\theoremstyle{definition}
\newtheorem{definition}[theorem]{Definition}

\newcommand{\seb}[1]{\textcolor[rgb]{0.00,0.00,1.00}{  #1}}
\begin{document}

\title[Navier boundary conditions and rough domains]
{The Stokes problem with Navier boundary conditions in irregular domains}

\author{Dominic Breit}
\address{Institute of Mathematics, TU Clausthal, Erzstra\ss e 1, 38678 Clausthal-Zellerfeld, Germany}
\email{dominic.breit@tu-clausthal.de}
\author{Sebastian Schwarzacher}
\address{Department of Mathematical Analysis,
	Faculty of Mathematics and Physics,
	Charles University,
	Sokolovská 83,
	186 75 Praha 8, Czech Republic}
\address{and}
\address{Department of Mathematics, 
	Analysis and Partial Differential Equations, 
	Uppsala University, 
	L\"agerhyddsv\"agen 1,
	752 37 Uppsala, Sweden}
\email{schwarz@karlin.mff.cuni.cz}



\date{\today}



\begin{abstract}

 We consider the steady Stokes equations supplemented with Navier boundary conditions including a non-negative friction coefficient. We prove maximal regularity estimates (including the prominent spaces $W^{1,p}$ and $W^{2,p}$ for $1<p<\infty$ for the velocity field) in bounded domains of minimal regularity. Interestingly, exactly one derivative more is required for the local boundary charts compared to the case of no-slip boundary conditions. We demonstrate the sharpness of our results by a propos examples.

\end{abstract}

\maketitle

\section{Introduction}

For a given forcing  $\bff:\Omega\rightarrow\R^{n}$ or  $\bfF:\Omega\rightarrow\R^{n\times n}$
we consider the Stokes system
 \begin{align}
 -\Delta\bfu+\nabla \pi=\bff,\quad \Div\bfu=0,&\label{3}
 \end{align}
in a bounded domain $\Omega\subset\R^n$, $n=2,3$.  
%
We aim at a maximal regularity theory, which includes estimates of the form
\begin{align}\label{eq:aim}
\|\nabla^2\bfu\|_{L^p(\Omega)}+\|\nabla\pi\|_{L^p(\Omega)}\lesssim \|\bff\|_{L^p(\Omega)}
\end{align}
or in case $\bff=\Div \bfF$ of type
\begin{align}\label{eq:aim'}
\|\nabla\bfu\|_{L^p(\Omega)}+\|\pi\|_{L^p(\Omega)}\lesssim \|\bfF\|_{L^p(\Omega)}
\end{align}
for $1<p<\infty$.
 A classical question concerns the validity of \eqref{eq:aim} under no-slip boundary conditions 
\begin{align}\label{eq:ns}
\bfu=0\quad\text{on}\quad\partial\Omega.
\end{align}
First results are attributed to Cattabriga \cite{Ca}, for an exhaustive picture and detailed references we refer to Galdi's book \cite[Chapter IV]{Ga}. The classical assumption here is that $\partial\Omega$ has a $C^2$-boundary for \eqref{eq:aim} and $C^1$-for \eqref{eq:aim'}. The latter result can be found in~\cite{BulBurSch16}, which includes also certain non-linear perturbations of the Stokes equation. However, many applications require a rougher boundary.
Under minimal assumptions concerning the boundary regularity a corresponding theory has been developed only very recently by the first author \cite{Br}. This is based on the theory of Sobolev multipliers, cf. the book by Maz'ya--Shaposhnikova \cite{MaSh}, which at generalizes respective assumptions that can be found in~\cite{LSU}. It has been employed before successfully to obtain regularity estimates for the Laplace equation in non-smooth domains.
As in the case of the Laplace equation (where necessity of this assumption is proved in \cite[Chapter 14]{MaSh}) the $W^{2,p}$ estimate \eqref{eq:aim} requires that the local boundary charts belong the the space of Sobolev multipliers (see Section \ref{sec:SM} for the precise definition and basic properties)
\begin{align}\label{eq:noslip}
\mathscr M^{2-1/p,p},
\end{align}
a subset of the trace space $W^{2-1/p,p}$. If the index of the space is sufficiently large (such that it is a multiplication algebra) one has that $\mathscr M^{2-1/p,p}\cong W^{2-1/p,p}$. Otherwise, more integrability is required for a Sobolev function to belong to the corresponding multiplier space cf.~Section \ref{sec:SM} below.

In several applications, the behaviour of the fluid close to the boundary is not adequately described by \eqref{eq:ns}. Hence different boundary conditions have been proposed in literature.
Probably the most prominent one was suggested by Navier in \cite{Na}. The boundary condition allows for slipping at the interface, which appears, when the surface is roughness~\cite{JM00}. Besides its importance in homogenization for rough boundaries partially constraint boundary conditions appear also in many other physical applications as free boundary problems, inflow outflow and more~\cite{BMR09}. 

The slipping is usually restricted by
a friction coefficient $\alpha:\partial\Omega\rightarrow[0,\infty)$ which measures the tendency of the fluid to slip over the boundary. In mathematical formulas the so-called Navier boundary conditions read as
\begin{align} \label{3a}
\bfu\cdot\bfn=0,\quad (\nabla\bfu\,\bfn)_{\bftau}+\alpha \bfu_\bftau=0,\quad\text{on}\quad\partial\Omega.
\end{align}
Here $\bfn$ denotes normal vector at $\partial\Omega$ and we set $\bfv_\bftau:=\bfv-(\bfv\cdot\bfn)\bfn$. Of special analytic interest is the perfect-slip case, when $\alpha\equiv 0$.

Despite its importance for application the literature on regularity is not very extensive here. Recent results are given by Acevedo Tapia et al. \cite{AACG} (see also \cite{AR}), where it has been shown, that \eqref{eq:aim'} holds for the system \eqref{3}, \eqref{3a} provided the boundary $\partial\Omega$ belongs to the class $C^{1,1}$. The space $C^{1,1}$ actually is the sharp limit space where so called BMO estimates are available~\cite{MS}. The papers \cite{KSW1,KSW} provide a parabolic counterpart of \eqref{eq:aim'} for the unsteady problem in a two-dimensional wedge.

In Theorems \ref{thm:stokessteady} and \ref{thm:stokessteadyF} we offer an exhaustive picture concerning
the maximal regularity theory for the Stokes system \eqref{3} with inhomogeneous Navier boundary conditions \eqref{3a}
in irregular domains in the framework of fractional Sobolev spaces. In particular, we prove that estimate \eqref{eq:aim} is true provided the local boundary charts of $\partial\Omega$ belong to $\mathscr M^{3-1/p,p}$ and \eqref{eq:aim'} holds if they lie in $\mathscr M^{2-1/p,p}$. Interestingly, this requires exactly one derivative more compared to the no-slip case as explained in 
\eqref{eq:noslip} above. This was already observed in \cite{MS} for the Schauder theory, where a $C^{1,\alpha}$ theory is sharply available, when the domain is $C^{2,\alpha}$. Indeed, as was observed there, sharpness can actually be more or less directly exported from the Laplace theory. 

\subsection{Sharpness of the boundary regularity.}
We comment now on the necessity of these conditions. We explain in detail the sharpness for \eqref{eq:aim}. However, we claim that similar arguments allow to show the sharpness also for the other cases, as  we can generally reduce the problem to the Laplace equation with Dirichlet boundary values, where the sharpness is known in \cite[Chapter 14]{MaSh}. The respective velocity satisfying Stokes equation with Navier-slip boundary values then possesses exactly one degree of differentiability less.

Following the second author's work \cite[Section 5]{MS} we consider the problem
\begin{align}
\label{eq:bilaplace}
\begin{aligned}
\Delta^2 w=-\mathrm{curl}\,\bff=-\mathrm{curl}\Div\bfF\quad&\text{in}\quad\Omega,\\
w=\Delta w=0\qquad\qquad\quad&\text{on}\quad\partial\Omega,
\end{aligned}
\end{align}
where $\Omega\subset\R^2$ and $\mathrm{curl}\, \bfv=-\partial_2 v^1+\partial_1 v^2$ for a vector field $\bfv:\Omega\rightarrow\R^2$. One easily checks, cf.  \cite[Section 5.6]{MS} for details, that
 $\bfu=(-\partial_2 w,\partial_1 w)^\top$ solves \eqref{3} with perfect slip boundary conditions, i.e., \eqref{3a} for $\alpha=0$. We make the particular choice
\begin{align*}
\Omega=\{(x,y)\in\R^2:\,y\geq \phi(x)\} \subset\R^2
\end{align*}
for a given function $\phi:\R\rightarrow\R$ and choose $\bff=0$ (or $\bfF=0$). By the Riemannian mapping theorem there exists a holomorphic function
\begin{align*}
z:\Omega\rightarrow\mathbb H:=\{(x,y)\in\R^2:\,y\geq0\}.
\end{align*}
Setting $w:=\mathrm{Im}(z)$, we clearly have $w=\Delta w=0$ on $\partial\Omega$.
Suppose now that $\bfu\in W^{2,p}_{\mathrm loc}$ for some $p>1$ and $\phi\in C^{1,1}(\R)$. 
Then, using that $\bfu=(-\partial_2 w,\partial_1 w)^\top$ we must have $w\in W^{3,p}_{\mathrm loc}$. Now \cite[Theorem 14.6.3]{MaSh} applies yielding $\phi\in W_{\mathrm loc}^{3-1/p,p}$. We provided now the elementary argument leading to this implication.
Since $w$ has zero boundary values by construction, the tangential derivative vanishes as well (see \cite[equ. (9.5.3)]{MaSh} for a rigorous proof), i.e.,
\begin{align*}
\mathrm{tr}\,\partial  (w\circ \bfPhi)+\mathrm{tr}\,\partial_y (w\circ \bfPhi)\partial \phi=0\quad\text{on}\quad\partial\Omega.
\end{align*}
Here $\bfPhi$ is an extension of $\phi$ (see \eqref{eq:Phi} below for details) and $\mathrm{tr}$ the trace operator related to $\partial\Omega$.
Noticing that $\mathrm{tr}\,\partial_y (w\circ \bfPhi)$ is strictly positive (by Hopf's maximum principle) this is equivalent to
\begin{align*}
\partial \phi=-\frac{\mathrm{tr}\,\partial  (w\circ \bfPhi)}{\mathrm{tr}\,\partial_y (w\circ \bfPhi)}\quad\text{on}\quad\partial\Omega.
\end{align*}
For $n=2$ and $p>1$ we have that $W^{2-1/p,p}_{\mathrm loc}$ is a multiplication algebra which implies $\phi\in W^{3-1/p,p}_{\mathrm loc}$ as desired. A similar argument can be applied if $\bfu\in W^{1,p}_{\mathrm loc}$ for some $p>2$ and $\phi\in C^{0,1}(\R)$
 implying  $\phi\in W^{2-1/p,p}_{\mathrm loc}$.
This verifies the sharpness of our assumptions concerning the boundary regularity.

\subsection{The structure of the paper.}
The structure of the paper is as follows. We present some preliminary material in the next section. In particular, we introduce the functional analytical framework including known results on Sobolev multipliers. Eventually, we discuss the parametrisation of domains by local charts. In Section \ref{sec:half} we consider the system \eqref{3}, \eqref{3a}. in the half space.
The bulk of the paper is Section \ref{sec:stokessteadyF} in which we prove the estimate for the problem in divergence form. In the subsequent section we reduce the problem in non-divergence form to it and obtain a corresponding theorem. In Appendix \ref{sec:AA} we give some details on the dual formulation for \eqref{3}, \eqref{3a}.
In Appendix \ref{sec:AB} we provide estimates for the Neumann problem for the Laplace equation in rough domains. They frequently serve as an auxiliary tool.

\section{Preliminaries}
\subsection{Conventions}
We write $f\lesssim g$ for two non-negative quantities $f$ and $g$ if there is a $c>0$ such that $f\leq\,c g$. Here $c$ is a generic constant which does not depend on the crucial quantities. If necessary we specify particular dependencies. We write $f\approx g$ if $f\lesssim g$ and $g\lesssim f$.
We do not distinguish in the notation for the function spaces between scalar- and vector-valued functions. However, vector-valued functions will usually be denoted in bold case and tensors by capital bold letters, i.e. we denote by $\bfF=(\bfF^1,....,\bfF^n)^\top$, where $\bfF^i:\setR^n\to \setR^n$.  We also use the convention that  $\Div\bfF:=(\Div \bfF^1,....,\Div \bfF^n)^\top$. Moreover, we define for a function $\bff:\setR^n\to \setR$, $\tilde{\bff}=(\bff_1,...,\bff_{n-1})^\top$
and $\tilde{\bfF}=(\bfF^1,....,\bfF^{n-1})$.
\subsection{Classical function spaces}
Let $\mathcal O\subset\R^m$, $m\geq 1$, be open.
We denote as usual by $L^p(\mathcal O)$ and $W^{k,p}(\mathcal O)$ for $p\in[1,\infty]$ and $k\in\mathbb N$ Lebesgue and Sobolev spaces over $\mathcal O$. For a bounded domain $\mathcal O$ the space $L^p_\perp(\mathcal O)$ denotes the subspace of  $L^p(\mathcal O)$ of functions with zero mean, that is $(f)_{\mathcal O}:=\dashint_{\mathcal O}f\dx:=\mathcal L^m(\mathcal O)^{-1}\int_{\mathcal O}f\dx=0$ with the $m$-dimensional Lebesgue measure $\mathcal L^m$.

 We denote by $W^{k,p}_0(\mathcal O)$ the closure of the smooth and compactly supported functions in $W^{k,p}(\mathcal O)$. If $\partial\mathcal O$ is regular enough, this coincides with the functions vanishing $\mathcal H^{m-1}$ -a.e. on $\partial\mathcal O$. 
 We also denote by $W^{-k,p}(\mathcal O)$ the dual of $W^{k,p}_0(\mathcal O)$. Furthermore, $W^{k,p}_n(\mathcal O)$ is defined as the vectorial functions from $W^{k,p}(\mathcal O)$ with $\bfu\cdot\bfn=0$ on $\partial\mathcal O$ (to be understood in the sense of traces).
  Finally, we consider subspaces
$W^{1,p}_{\Div}(\mathcal O)$, $W^{1,p}_{0,\Div}(\mathcal O)$ and $W^{1,p}_{n,\Div}(\mathcal O)$ of divergence-free vector fields which are defined accordingly. The space $L^p_{\Div}(\mathcal O)$ is defined as the closure of the smooth and compactly supported solenoidal functions in $L^p(\mathcal O).$ We will use the shorthand notations $L^p $ and $W^{k,p} $.

Last we introduce for unbounded domains $\mathcal O$ the homogeneous Sobolev spaces $\mathcal D^{k,p}(\mathcal O)$ as the set of all locally $p$-integrable functions with finite $\|\nabla^k\cdot\|_{L^p(\mathcal O)}$-semi norm. Similar to the above $\mathcal D^{k,p}_n(\mathcal O)$ is defined as the set of vectorial functions from $\mathcal D^{k,p}(\mathcal O)$ with $\bfu\cdot\bfn=0$ on $\partial\mathcal O$
along with its solenoidal variant $\mathcal D_{n,\Div}^{k,p}(\mathcal O)$, where we only take solenoidal functions from $\mathcal D^{k,p}(\mathcal O)$. The spaces  $\mathcal D^{-k,p}(\mathcal O)$ and  $\mathcal D^{-k,p}_n(\mathcal O)$ for $k\in\N$ are defined as the corresponding dual spaces.

\subsection{Fractional differentiability and Sobolev mulitpliers}\label{sec:SM}
For $p\in[1,\infty)$ the fractional Sobolev space (Sobolev-Slobodeckij space) with differentiability $s>0$ with $s\notin\mathbb N$ will be denoted by $W^{s,p}(\mathcal O)$. For $s>0$ we write $s=\lfloor s\rfloor+\lbrace s\rbrace$ with $\lfloor s\rfloor\in\N_0$ and $\lbrace s\rbrace\in(0,1)$.
 We denote by $W^{s,p}_0(\mathcal O)$ the closure of the smooth and compactly supported functions in $W^{1,p}(\mathcal O)$. For $s>\frac{1}{p}$ this coincides with the functions vanishing $\mathcal H^{m-1}$ -a.e. on $\partial\mathcal O$ provided $\partial\mathcal O$ is regular enough. We also denote by $W^{-s,p}(\mathcal O)$ for $s>0$ the dual of $W^{s,p}_0(\mathcal O)$. Similar to the case of unbroken differentiabilities above we use the shorthand notation $W^{s,p} $. The homogenous space
$\mathcal D^{s,p}(\mathcal O)$ for $s=\lfloor s\rfloor+\lbrace s\rbrace$ is defined via the semi-norm $\|\nabla^{\lfloor s\rfloor}\cdot\|_{W^{\lbrace s\rbrace,p}}$. Spaces with $s<0$ are defined as the corresponding dual spaces.

We will denote by $\bfB^s_{p,q}(\R^m)$ the standard Besov spaces on $\R^m$ with differentiability $s>0$, integrability $p\in[1,\infty]$ and fine index $q\in[1,\infty]$. They can be defined (for instance) via Littlewood-Paley decomposition leading to the norm $\|\cdot\|_{\bfB^s_{p,q}(\R^m)}$. 
 We refer to \cite{RuSi} and \cite{Tr,Tr2} for an extensive picture. 
 The Besov spaces $\bfB^s_{p,q}(\mathcal O)$ for a bounded domain $\mathcal O\subset\R^m$ are defined as the restriction of functions from $\bfB^s_{p,q}(\R^m)$, that is
 \begin{align*}
 \bfB^s_{p,q}(\mathcal O)&:=\{f|_{\mathcal O}:\,f\in \bfB^s_{p,q}(\R^m)\},\\
 \|g\|_{\bfB^s_{p,q}(\mathcal O)}&:=\inf\{ \|f\|_{\bfB^s_{p,q}(\R^m)}:\,f|_{\mathcal O}=g\}.
 \end{align*}
 If $s\notin\mathbb N$ and $p\in[1,\infty)$ we have $\bfB^s_{p,p}(\mathcal O)=W^{s,p}(\mathcal O)$.
 
In accordance with \cite[Chapter 14]{MaSh} the Sobolev multiplier norm  is given by
\begin{align}\label{eq:SoMo}
\|\varphi\|_{\mathscr M^{s,p}(\mathcal O)}:=\sup_{\bfv:\,\|\bfv\|_{W^{s-1,p}(\mathcal O)}=1}\|\nabla\varphi\cdot\bfv\|_{W^{s-1,p}(\mathcal O)},
\end{align}
where $p\in[1,\infty]$ and $s\geq1$.
The space $\mathscr M^{s,p}(\mathcal O)$ of Sobolev multipliers is defined as those objects for which the $\mathscr M^{s,p}(\mathcal O)$-norm is finite. For $\delta>0$ we denote by $\mathscr M^{s,p}(\mathcal O)(\delta)$ the subset of functions from
$\mathscr M^{s,p}(\mathcal O)$ with $\mathscr M^{s,p}(\mathcal O)$-norm not exceeding $\delta$. By mathematical induction with respect to $s$ one can prove for Lipschitz-continuous functions $\varphi$ that membership to $\mathscr M^{s,p}(\mathcal O)$  in the sense of \eqref{eq:SoMo} implies that
\begin{align}\label{eq:SoMo'}
\|\varphi\|_{\mathscr M^{s,p}_{\tt or}(\mathcal O)}:=\sup_{w:\,\|w\|_{W^{s,p}(\mathcal O)}=1}\|\varphi \,w\|_{W^{s,p}(\mathcal O)}<\infty.
\end{align}
The quantity \eqref{eq:SoMo'}, even defined for all $s\geq0$, also serves as customary definition of the Sobolev multiplier norm in the literature but \eqref{eq:SoMo} is more suitable for our purposes. 
Note that in our applications we always assume that the functions in question are Lipschitz continuous such that the implication above is given. We also consider multipliers between Sobolev spaces $W^{s_2,p}$ and $W^{s_1,p}$ with different differentiabilities\footnote{Clearly, one can also consider different integrabilities but this is not needed for our purposes.} $s_2\geq s_1$
\begin{align}\label{eq:SoMo''}
\|\varphi\|_{\mathscr M^{s_1,p}_{s_2,p}(\mathcal O)}:=\sup_{w:\,\|w\|_{W^{s_2,p}(\mathcal O)}=1}\|\varphi \,w\|_{W^{s_1,p}(\mathcal O)}<\infty,
\end{align}
where even negative $s_1$ can be considered.
This leads to a multiplier space $\mathscr M^{s_1,p}_{s_2,p}(\mathcal O)$, where $\mathscr M^{s,p}_{s,p}(\mathcal O)=\mathscr M^{s,p}_{\tt or}(\mathcal O)$ for $s\geq0$.

Let us finally collect some useful properties of Sobolev multipliers. By Sobolev's embedding
one easily checks that 
a function belongs to $\mathscr M^{s,p}(\R^m)$
provided that one of the following conditions holds for some $\varepsilon>0$:
\begin{itemize}
\item $p(s-1)<m$ and $\phi\in \bfB^{s+\varepsilon}_{\varrho,p}(\R^{m})$ with $\varrho\in\big[\frac{pm}{p(s-1)-1},\infty\big]$;
\item $p(s-1)=m$ and $\phi\in\bfB^{s+\varepsilon}_{\varrho,p}(\R^m)$ with $\varrho\in(p,\infty]$.
\end{itemize}
In some cases, important for the parametrisation of the boundary of an $n$-dimensional domain, this statement can be sharpened.
By \cite[Corollary 14.6.2]{MaSh} we have for $\phi\in\bfB^{s}_{\varrho,p}(\R^{n-1})$ compactly supported and $\delta\ll1$ that
\begin{align}\label{eq:MSa}
\|\phi\|_{\mathscr M^{s,p}(\R^{n-1})}\leq\,c(\|\phi\|_{\bfB^{s}_{\varrho,p}(\R^{n-1})})\delta,
\end{align}
provided that $\|\nabla\phi\|_{L^{\infty}(\R^{n-1})}\leq\delta$, $s=l-1/p$ for some $l\in\N$ and one of the following conditions holds:
\begin{itemize}
\item $p(l-1)<n$ and $\phi\in \bfB^{s}_{\varrho,p}(\R^{n-1})$ with $\varrho\in\big[\frac{p(n-1)}{p(l-1)-1},\infty\big]$;
\item $p(l-1)=n$ and $\phi\in\bfB^{s}_{\varrho,p}(\R^{n-1})$ with $\varrho\in(p,\infty]$.
\end{itemize}
By \cite[Corollary 4.3.8]{MaSh} it holds
\begin{align}\label{eq:MSb}
\|\phi\|_{\mathscr M^{s,p}(\R^{m})}\approx
\|\nabla\phi\|_{W^{s-1,p}(\R^{m})} 
\end{align}
for $p(s-1)>m$. 
If $\mathcal O$ is a Lipschitz domain
with diameter $r$ and $ps>m$ we have
\begin{align}\label{eq:MSc}
\|\phi\|_{\mathscr M^{s,p}_{\tt or}(\mathcal O)}\approx
r^{s-m/p}\|\phi\|_{W^{s,p}(\mathcal O)},
\end{align}
cf.~\cite[Theorem 9.6.1. (ii)]{MaSh}.

Finally, we note the following rule about the composition with Sobolev multipliers which is a consequence of \cite[Lemma 9.4.1]{MaSh}. For open sets $\mathcal O_1,\mathcal O_2\subset\R^m$, $u\in W^{s,p}(\mathcal O_2)$ and a Lipschitz continuous function $\bfphi:\mathcal O_1\rightarrow\mathcal O_2$ with Lipschitz continuous inverse and $\bfphi\in \mathscr M^{s,p}(\mathcal O_1)$ we have
\begin{align}\label{lem:9.4.1}
\|u\circ\bfphi\|_{W^{s,p}(\mathcal O_1)}\lesssim \|u\|_{W^{s,p}(\mathcal O_2)}
\end{align}
with constant depending on $\bfphi$. Using Lipschitz continuity
of $\bfphi$ and $\bfphi^{-1}$, estimate \eqref{lem:9.4.1} is obvious for $s\in(0,1]$. The general case can be proved by mathematical induction with respect to $s$. Clearly,
inequality \eqref{lem:9.4.1} also holds for $s\in(-1,1)$ if $\bfphi$ is a  Lipschitz function. 

\subsection{Parametrisation of domains}\label{sec:para}
We follow the presentation from \cite[Section 3]{Br}.
 Let $\Omega\subset\R^n$ be a bounded open set.
We assume that $\partial{\Omega}$ can be covered by a finite
number of open sets $\mathcal U^1,\dots,\mathcal U^\ell$ for some $\ell\in\mathbb N$, such that
the following holds. For each $j\in\{1,\dots,\ell\}$ there is a reference point
$y^j\in\R^n$ and a local coordinate system $\{e^j_1,\dots,e_n^j\}$ (which we assume
to be orthonormal and set $\mathcal Q_j=(e_1^j|\dots |e_n^j)\in\mathbb R^{n\times n}$), a function
$\varphi_j:\mathbb R^{n-1}\rightarrow\mathbb R$
and $r_j>0$
with the following properties:
\begin{enumerate}[label={\bf (A\arabic{*})}]
\item\label{A1} There is $h_j>0$ such that
$$\mathcal U^j=\{x=\mathcal Q_jz+y^j\in\mathbb R^n:\,z=(z',z_n)\in\R^n,\,|z'|<r_j,\,
|z_n-\varphi_j(z')|<h_j\}.$$
\item\label{A2} For $x\in\mathcal U^j$ we have with $z=\mathcal Q_j^\top(x-y^j)$
\begin{itemize}
\item $x\in\partial{\Omega}$ if and only if $z_n=\varphi_j(z')$;
\item $x\in{\Omega}$ if and only if $0<z_n-\varphi_j(z')<h_j$;
\item $x\notin{\Omega}$ if and only if $0>z_n-\varphi_j(z')>-h_j$.
\end{itemize}
\item\label{A3} We have that
$$\partial{\Omega}\subset \bigcup_{j=1}^\ell\mathcal U^j,\text{ with  $\mathcal U^j$ having finite overlap.}$$
\end{enumerate}
In other words, for any $x_0\in\partial{\Omega}$ there is a neighborhood $U$ of $x_0$ and a function $\varphi:\mathbb R^{n-1}\rightarrow\mathbb R$ such that after translation and rotation\footnote{By translation via $y_j$ and rotation via $\mathcal Q_j$ we can assume that $x_0=0$ and that the outer normal at~$x_0$ is pointing in the negative $x_n$-direction.}
 \begin{align}\label{eq:3009}
 U \cap {\Omega} = U \cap G,\quad G = \set{(x',x_n)\in \R^n \,:\, x' \in \R^{n-1}, x_n > \varphi(x')}.
 \end{align}
 The regularity of $\partial{\Omega}$ will be described by means of local coordinates as just described.
 \begin{definition}\label{def:besovboundary}
 Let ${\Omega}\subset\R^n$ be a bounded domain, $s\geq 1$ and $1\leq p\leq\infty$. We say that $\partial{\Omega}$ belongs to the class $\mathscr M^{s,p}$ if there is $\ell\in\mathbb N$ and functions $\varphi_1,\dots,\varphi_\ell\in\mathscr M^{s,p}(\mathbb R^{n-1})$ satisfying \ref{A1}--\ref{A3}.
 \end{definition}
Clearly, a similar definition applies for
various other function spaces. In particular, we say that $\partial{\Omega}$ belongs to the class $\mathscr M^{s,p}( \delta)$ for some $\delta>0$ (or $B^s_{p,q}$ with $1\leq q\leq \infty$), if there exist $\mathcal U^1,\dots,\mathcal U^\ell$ with maximal overlap $m$ and $\varphi_1,\dots,\varphi_\ell\in\mathscr M^{s,p}(\mathbb R^{n-1})(\frac{\delta}{m})$ (or in $B^s_{p,q}(\R^{n-1})$).
Also, we speak about a Lipschitz boundary (or sometimes a $C^{1,\alpha}$-boundary with $\alpha\in(0,1)$) by requiring that $\varphi_1,\dots,\varphi_\ell\in W^{1,\infty}(\mathbb R^{n-1})$ (or $\varphi_1,\dots,\varphi_\ell\in C^{1,\alpha}(\mathbb R^{n-1})$). We say that the local Lipschitz constant of $\partial{\Omega}$, denoted by $\mathrm{Lip}(\partial{\Omega})$, is (smaller or) equal to some number $L>0$ provided the Lipschitz constants of $\varphi_1,\dots,\varphi_\ell$ are not exceeding $L$. Our main result depends on the assumption of a sufficiently small local Lipschitz constant. 
It holds, for instance, if the regularity
of $\partial{\Omega}$ is better than Lipschitz (such as $C^{1,\alpha}$ for some $\alpha>0$). 

In order to describe the behaviour of functions defined in ${\Omega}$ close to the boundary we need to extend the functions $\varphi_1,\dots,\varphi_\ell$  from \ref{A1}--\ref{A3} to the half space
$\mathbb H := \set{z = (z',z_n)\,:\, z_n > 0}$. Hence we are confronted with the task of extending a function~$\phi\,:\, \R^{n-1}\to \R$ to a mapping $\bfPhi\,:\, \mathbb H \to \R^n$ that maps the 0-neighborhood in~$\mathbb H$ to the $x_0$-neighborhood in~${\Omega}$. The mapping $(z',0) \mapsto (z',\phi(z'))$ locally maps the boundary of~$\mathbb H$ to the one of~$\partial {\Omega}$. We extend this mapping using the extension operator of Maz'ya and Shaposhnikova~\cite[Section 9.4.3]{MaSh}. Let $\zeta \in C^\infty_c(B_1(0'))$ with $\zeta \geq 0$ and $\int_{\R^{n-1}} \zeta(x')\dx'=1$. Let $\zeta_t(x') := t^{-(n-1)} \zeta(x'/t)$ denote the induced family of mollifiers. We define the extension operator 
\begin{align*}
  (\mathcal{T}\phi)(z',z_n)=\int_{\R^{n-1}} \zeta_{z_n}(z'-y')\phi(y')\dy',\quad (z',z_n) \in \mathbb H,
\end{align*}
where~$\phi:\R^n\to \R$ is a Lipschitz function with Lipschitz constant~$K$.
Then the estimate 
\begin{align}\label{est:ext}
  \norm{\nabla (\mathcal{T} \phi)}_{\bfB_{\rho,q}^{s}(\setR^{n})}\le c\norm{\nabla \phi}_{\bfB_{\rho,q}^{s-\frac 1 p}(\setR^{n-1})}
\end{align}
for $s\geq1+\frac{1}{p}$ and $\rho,q\in[1,\infty]$ follows from~\cite[Section 9.4.3]{MaSh}. Moreover, \cite[Theorem 8.8.1]{MaSh} yields
\begin{align}\label{eq:MS}
\|\mathcal T\phi\|_{\mathscr M^{s,p}(\mathbb H)}\lesssim \|\phi\|_{\mathscr M^{s-1/p,p}(\R^{n-1})}
\end{align}
for $s\geq1+\frac{1}{p}$ and $p\in[1,\infty)$.
It is shown in \cite[Lemma 9.4.5]{MaSh} that (for sufficiently large~$N$, i.e., $N \geq c(\zeta) K+1$) the mapping
\begin{align*}
  \alpha_{z'}(z_n) \mapsto N\,z_n+(\mathcal{T} \phi)(z',z_n)
\end{align*}
is for every $z' \in \setR^{n-1}$ one to one and the inverse is Lipschitz with its gradient
bounded by $(N-K)^{-1}$.
Now, we define the mapping~$\bfPhi\,:\, \mathbb H \to \R^n$ as a rescaled version of the latter one by setting
\begin{align}\label{eq:Phi}
  \bfPhi(z',z_n)
  &:=
    \big(z',
    \alpha_{z'}(z_n)\big) = 
    \big(z',
    \,z_n + (\mathcal{T} \phi)(z',z_n/K)\big).
\end{align}
Thus, $\bfPhi$ is one-to-one (for sufficiently large~$N=N(K)$) and we can define its inverse $\bfPsi := \bfPhi^{-1}$.
The mapping $\bfPhi$ has the Jacobi matrix of the form
\begin{align}\label{J}
  J = \nabla \bfPhi = 
  \begin{pmatrix}
    \mathbb I_{(n-1)\times (n-1)}&0
    \\
    \partial_{z'} (\mathcal{T}  \phi)& 1+ 1/N\partial_{z_n}\mathcal{T}  \phi
  \end{pmatrix}.
\end{align}
Since 
$\abs{\partial_{z_n}\mathcal{T}  \phi} \leq K$, we have \begin{align}\label{eq:detJ}\frac{1}{2} < 1-K/N \leq \abs{\det(J)} \leq 1+K/N\leq 2\end{align}
using that $N$ is large compared to~$K$. Finally, we note the implication
\begin{align} \label{eq:SMPhiPsi}
\bfPhi\in\mathscr M^{s,p}(\mathbb H))\,\,\Rightarrow \,\,\bfPsi\in\mathscr M^{s,p}(\mathbb H),
\end{align}
which holds, for instance, if $\bfPhi$ is Lipschitz continuous, cf. \cite[Lemma 9.4.2]{MaSh}. In fact, one can prove \eqref{eq:SMPhiPsi} with the help of \eqref{lem:9.4.1} and \eqref{eq:detJ}.

\subsection{The problem in the half space}
\label{sec:half}
In this subsection we study the Stokes equations in the half space $$\mathbb H:=\{x=(x',x_n)\in\R^n:\,\,x_n>0\}$$ with perfect slip boundary conditions on $$\partial\mathbb H=\{(x',0)\in\R^n:\,x'\in\R^{n-1}\},$$ that is
\begin{align}\label{eq:Stokeshalf}
\begin{aligned}
&-\Delta \bfu+\nabla\pi=\bff+\Div\bfF,\quad\Div\bfu=h,\quad&\text{in}\quad\mathbb H,\\
&u^n=\mathfrak g,\quad \partial_n \tilde{\bfu}+\tilde\bfF_n=\GG,\quad&\text{on}\quad\partial\mathbb H,
\end{aligned}
\end{align}
for some given tensor $\bfF$ and functions $\bff$, $h$, $\mathfrak g$ with $\int_{\mathbb H}h\dx=\int_{\partial\mathbb H}\mathfrak g\,\dd\mathcal H^{n-1}$ and $\GG$. 
 
We may assume that $\mathfrak g=0=h$. Otherwise, we solve the Neumann problem
\begin{align*}
\Delta\vartheta=h\quad\text {in}\quad\mathbb H,\quad\partial_n\vartheta=\mathfrak g\quad\text{on}\quad\partial\mathbb H,
\end{align*}
and replace $\bfu$ by $\bfu-\nabla\vartheta$. It is classical that
\begin{align*}
 \|\nabla\vartheta\|_{L^p(\mathbb H)}\lesssim \|\mathfrak g\|_{W^{1-1/p,p}(\partial\mathbb H)}+\|h\|_{W^{-1,p}(\mathbb H)},\quad  \|\nabla^2\vartheta\|_{L^p(\mathbb H)}\lesssim \|\mathfrak g\|_{W^{2-1/p,p}(\partial\mathbb H)}+\|h\|_{L^p(\mathbb H)},
\end{align*}
provided $\mathfrak g, h$ belong to the corresponding spaces. 
Finally, in order to reduce the problem to the case $\GG-\tilde{\bfF}_n =0$ we employ
we employ an extension. For $\GG^i-\tilde{\bfF}_n^i\in W^{-1/p,p}(\mathbb H)$, $i=1,2$, there functions
$\boldsymbol{\mathfrak K}^i\in L^p(\mathbb H)$ with
$(\boldsymbol{\mathfrak K}^i)^n=\mathfrak G^i$.
In particular, it holds
\begin{align}\label{eq:09.10}
 \|\boldsymbol{\mathfrak K}^i\|_{L^p(\mathbb H)}+ \|\Div\boldsymbol{\mathfrak K}^i\|_{L^p(\mathbb H)}\lesssim \|\GG-\tilde{\bfF}_n\|_{W^{-1/p,p}(\partial\mathbb H)},
\end{align}
see \cite[Chapter I]{Te} and \cite[page 26]{So}. Now we solve the auxiliary problem
\begin{align*}
-\Delta {\mathfrak u}^i=(\boldsymbol{\mathfrak K}^i)^n\quad\text {in}\quad\mathbb H,\quad {\mathfrak u}^i=0\quad\text{on}\quad\partial\mathbb H,\quad i=1.,\dots, n-1,
\end{align*}
and replace $\bfu$ by 
\begin{align*}
\bfw=\begin{pmatrix}
u^1+\partial_n\mathfrak u^1\\
\vdots\\
u^{n-1}+\partial_n\mathfrak u^{n-1}\\
u^n-\sum_{i=1}^{n-1}\partial_i\mathfrak u^i
\end{pmatrix},
\end{align*}
which is divergence free. Further one easily checks that this function is divergence-free, and satisfies $\partial_n\tilde\bfw=0$ and $w_n=0$ on $\partial\mathbb H$. Moreover, it holds by well-known estimates for the Laplace equation
and \eqref{eq:09.10}
\begin{align*}
\|\nabla\bfu\|_{L^p(\mathbb H)}&\lesssim \|\nabla\bfw\|_{L^p(\mathbb H)}+\sum_{i=1}^{n-1}\|\nabla^2\mathfrak u^i\|_{L^p(\mathbb H)}\\
&\lesssim \|\nabla\bfw\|_{L^p(\mathbb H)}+\sum_{i=1}^{n-1}\|(\mathfrak K^i)^n\|_{L^p(\mathbb H)}\lesssim \|\nabla\bfw\|_{L^p(\mathbb H)}+\|\GG-\tilde{\bfF}_n\|_{W^{-1/p,p}(\partial\mathbb H)},
\end{align*}
and similarly for higher order derivatives. Hence we moved all data into one single right hand side in divergence form, which has the appropriate boundary values.

The weak formulation for the problem in divergence-form with $h=\mathfrak g=0$ and $\GG=\tilde{\bfF}_n=0$ we are looking for a function $\bfu$ such that
\begin{align}\label{eq:weak}
\int_{\mathbb H}\boldsymbol{\varepsilon}(\bfu):\boldsymbol{\varepsilon}(\bfphi)\dx=\int_{\mathbb H}\nabla\bfu:\nabla\bfphi\dx
=
-\int_{\mathbb H}\bfF:\nabla\bfphi\dx
\end{align}
for all $\bfphi\in \mathcal D^{1,2}_{n,\Div}(\mathbb H)$, where $\boldsymbol{\varepsilon}=\frac{1}{2}(\nabla+\nabla^\top)$ denotes the symmetric gradient.
A solution $\bfu$ to \eqref{eq:weak} exists in the homogeneous space $\mathcal D^{1,2}_{n,\Div}(\mathbb H)$ provided $\bfF\in L^2(\mathbb H)$
 as can be shown by Riesz' theorem in Hilbert spaces. The velocity field is unique. The pressure function $\pi$ from \eqref{eq:Stokeshalf} can be recovered in the space $L^2(\mathbb H)$ and is unique up to a constant. 
One easily realizes that this system can be reflected to the whole space. Indeed by defining 
 \begin{align*}
 \tilde{\bfw}(x',z)&=\tilde{\bfu}(x',z)\text{ when }z\geq 0\text{ and }\tilde{\bfw}(x',z)=\tilde{\bfu}(x',-z)\text{ when }z<0
 \\
w_n(x',z)&=u_n(x',z)\text{ when }z\geq 0\text{ and }w_n(x',z)=-u(x',-z)\text{ when }z<0
\\
\mathfrak q(x',z)&=\pi(x',z)\text{ when }z\geq 0\text{ and }\mathfrak q(x',z)=\pi(x',-z)\text{ when }z<0
\\
\tilde{\bfG}_i(x',z)&=\tilde{\bfF}_i(x',z)\text{ when }z\geq 0\text{ and }\tilde{\bfG}_i(x',z)=\tilde{\bfF}_i(x',-z)\text{ when }z<0\text{ for }i\neq n
 \\
\tilde{\bfG}_n(x',z)&=\tilde{\bfF}_n(x',z)\text{ when }z\geq 0\text{ and }\tilde{\bfG}_n(x',z)=-\tilde{\bfF}_n(x',-z)\text{ when }z<0
\\
\bfG_i^n(x',z)&={\bfF}_i^n(x',z)\text{ when }z\geq 0\text{ and }\tilde{\bfG}_i^n(x',z)=-\tilde{\bfF}_i^n(x',-z)\text{ when }z<0\text{ for }i\neq n
\\
\bfG_n^n(x',z)&={\bfF}_n^n(x',z)\text{ when }z\geq 0\text{ and }\tilde{\bfG}_n^n(x',z)=\tilde{\bfF}_n^n(x',-z)\text{ when }z<0
 \end{align*}
 we find formally that
 \[
 -\Delta \bfw +\nabla \mathfrak q=\Div \bfG\text{ and }\quad \Div \bfw=0\text{ in }\setR^n.
 \]
  Moreover, if $\bfF$ was continuous on $\mathbb{H}$, then $\bfG$ is continuous at
$\setR^n$. It is also true for weak solutions. Hence we can use the classical theory, see \cite[Chapter IV]{Ga}.
We obtain the following result using \cite[Chapter IV]{Ga} and real interpolation.
\begin{theorem}\label{thm:StokesH}
 \begin{enumerate}
\item[(a)]
Let $q\in(1,\infty)$, $s\in[2,\infty)$ and  
suppose that we have $\bff\in W^{s-2,q}\cap \mathcal D_n^{-1,2}(\mathbb H)$, $h\in L^1\cap\mathcal D^{s,q}(\mathbb H)$, $\mathfrak g\in L^1\cap W^{s-1/q,q}(\partial\mathbb H)$  with $\int_{\mathbb H}h\dx=\int_{\partial\mathbb H}\mathfrak g\,\dd\mathcal H^{n-1}$ and $\GG\in W^{s-1-1/q,q}(\partial\mathbb H)$.
There is a unique solution
 $(\bfu,\pi)\in \mathcal D^{s,q}_{n,\Div}(\mathbb H)\times \mathcal D^{s-1,q}(\mathbb H)$ to \eqref{eq:Stokeshalf} and we have
\begin{align}\label{eq:mainH}
\begin{aligned}
\|\nabla^2\bfu\|_{W^{s-2,q}(\mathbb H)}+\|\nabla\pi\|_{W^{s-2,q}({\mathbb H})}&\lesssim\|\bff\|_{W^{s-2,p}({\mathbb H})}+\|h\|_{W^{s-1,q}(\partial\mathbb H)}\\
&+\|\mathfrak g\|_{W^{s-1/q,q}(\partial\mathbb H)}+\|\GG\|_{W^{s-1-1/q,q}(\partial\mathbb H)}.
\end{aligned}
\end{align}
\item[(b)]
Let $q\in(1,\infty)$, $s\in[1,\infty)$ and  
suppose that we have $\bfF\in W^{s-1,q}(\mathbb H)$, $h\in L^1\cap\mathcal D^{s,q}(\mathbb H)$, $\mathfrak g\in L^1\cap W^{s-1/q,q}(\partial\mathbb H)$  with $\int_{\mathbb H}h\dx=\int_{\partial\mathbb H}\mathfrak g\,\dd\mathcal H^{n-1}$ and $\GG\in W^{s-1-1/q,q}(\partial\mathbb H)$.
There is a unique solution
 $(\bfu,\pi)\in \mathcal D^{s,q}_{n,\Div}(\mathbb H)\times \mathcal D^{s-1,q}(\mathbb H)$ to \eqref{eq:Stokeshalf} and we have
\begin{align}\label{eq:mainHF}
\begin{aligned}
\|\nabla\bfu\|_{W^{s-1,q}(\mathbb H)}+\|\pi\|_{W^{s-1,q}({\mathbb H})}&\lesssim\|\bfF\|_{W^{s-1,p}({\mathbb H})}+\|h\|_{W^{s-1,q}(\partial\mathbb H)}\\
&+\|\mathfrak g\|_{W^{s-1/q,q}(\partial\mathbb H)}+\|\GG\|_{W^{s-1-1/q,q}(\partial\mathbb H)}.
\end{aligned}
\end{align}
\end{enumerate}
\end{theorem}

\section{The problem in divergence form}\label{sec:stokessteadyF}
In this section we consider the steady Stokes system
\begin{align}\label{eq:Stokesdiv}
\begin{aligned}
\Delta \bfu-\nabla\pi=-\Div\bfF,\quad\Div\bfu=h,\quad\text{in}\quad\Omega\\
\bfu\cdot\bfn=\mathfrak g,\quad (\nabla\bfu\,\bfn+\bfF\,\bfn)_{\bftau}+\alpha \bfu_\bftau=\GG,\quad\text{on}\quad\partial\Omega,
\end{aligned}
\end{align}
in a domain ${\Omega}\subset\R^n$, where $\bfv_\bftau:=\bfv-(\bfv\cdot\bfn)\bfn$. The result given in the following theorem is a maximal regularity estimate for the solution in terms of the right-hand side under minimal assumption on the regularity of $\partial\Omega$ (the corresponding multiplier spaces are introduced in Section \ref{sec:SM}). We start with the case, where $W^{s,p}(\Omega)\hookrightarrow W^{1,2}(\Omega)$ and give a correspnonding result for the remaining parameters below in Corollary \ref{cor:stokessteadyF}.   
\begin{theorem}\label{thm:stokessteadyF}
Let $n=2,3$, $p\in(1,\infty)$ and $s\geq1$ such that $n\big(\frac{1}{p}-\frac{1}{2}\big)+1\leq s$. Assume that
$\alpha:\partial\Omega\rightarrow[0,\infty)$ is a measurable function belonging to the class $\alpha\in \mathscr M_{s-1/p,p}^{s-1-1/p,p}$.
Let $\Omega$ be a bounded Lipschitz domain and suppose
that either
\begin{enumerate}
\item[(a)]  $p(s-1)\leq n$ and $\partial\Omega$ belongs to the class
$\mathscr M^{s+1-1/p,p}(\delta)$ for some sufficiently small $\delta>0$ or
\item[(b)] $p(s-1)> n$ and $\partial\Omega$ belongs to the class
$W^{s+1-1/p,p}$. 
\end{enumerate}
For any $\bfF\in W^{s-1,p}({\Omega})$ , $h\in W^{s-1,p}(\Omega)$, $\mathfrak g\in W^{s-1/p,p}(\partial\Omega)$ with $\int_{\Omega}h\dx=\int_{\partial\Omega}\mathfrak g\,\dd\mathcal H^{n-1}$ and $\GG\in W^{s-1-1/p,p}(\partial\Omega)$
there is a unique solution to \eqref{eq:Stokesdiv} and we have
\begin{align}\label{eq:maindiv}
\begin{aligned}
\|\bfu\|_{W^{s,p}(\Omega)}+\|\pi\|_{W^{s-1,p}(\Omega)}&\lesssim\|\bfF\|_{W^{s-1,p}(\Omega)}+\|h\|_{W^{s-1,p}(\Omega)}\\&+\|\mathfrak g\|_{W^{s-1/p,p}(\partial\Omega)}+\|\GG\|_{W^{s-1-1/p,p}(\partial\Omega)}.
\end{aligned}
\end{align}
The constant in \eqref{eq:maindiv} depends on the Lipschitz constants as well as the
$\mathscr M^{s+1-1/p,p}$- or $W^{s+1-1/p,p}$-norms of the local charts in the parametrisation of  $\partial\Omega$.
\end{theorem}
\begin{remark}\label{rem:main2}This relationship between Sobolev multipliers and Besov spaces can be seen from \eqref{eq:MSa} and \eqref{eq:MSb}. 
 Suppose that $\Omega$ is a $\bfB^{\sigma+\varepsilon-1/p}_{\varrho,p}$-domain,
where $s+1=\sigma\geq1$ and $\varepsilon>0$,
\begin{align}\label{eq:SMp}
\varrho\geq p\quad\text{if}\quad p(\sigma-1)= n,\quad \varrho\geq \tfrac{p(n-1)}{p(\sigma-1)-1}\quad\text{if}\quad p(\sigma-1)< n,
\end{align}
 with locally small Lipschitz constant.
Then we have that $\partial\Omega$ belongs to the class
$\mathscr M^{\sigma-1/p,p}(\delta)$ for any $\delta>0$. If $\sigma\in\N$ and the Lipschitz constant of $\partial\Omega$ is locally small, we can include the case $\varepsilon=0$. 
In particular, if $s=\sigma-1=1$ and $p\leq n$ we require that $\partial\Omega$ has a locally small Lipschitz constant and belongs to the class
$\bfB^{2-1/p}_{\varrho,p}$, where $\varrho$ is given by \eqref{eq:SMp}. In particular it suffices that $\partial\Omega\in W^{2-\frac{1}{q_1},q_2}$ for some $q_1\leq p$ and $q_2>p$ and possesses locally a small Lipschitz constant.

In case $s=1$ and $p>n$ we find that the optimal boundary regularity is $\bfB^{2-1/p}_{p,p}=W^{2-\frac{1}{p},p}$. Please note that in this case  $C^{1,1}\subset W^{2-\frac{1}{p},p}\subset C^{1,\frac{p-n}{p}}$, locally.
\end{remark}
\begin{remark} 
If $p>2-2/n$, one easily checks that $L^{r}\hookrightarrow\mathscr M_{1-1/p,p}^{-1/p,p}$, where
\begin{align*}
r=\frac{(n-1)p}{np+2-2n}
\end{align*}
Hence Theorem \ref{thm:stokessteadyF} applies for $\alpha\in L^r(\partial\Omega)$.
\end{remark}
\begin{remark}
In the setting of Theorem \ref{thm:stokessteadyF} the localization sets $\mathcal U^1,\dots,\mathcal U^\ell$ can always be chosen such that their diameter and the Lipschitz constant of the $\varphi_j$'s are uniformly small with a fixed number of overlaps. This implies that one can assume, without loss of generality, that the multiplier norm is small on all orders. Indeed, any lower order term can be shown to be small by some interpolation argument involving the diameter or the Lipschitz constant. In particular, the respective assumptions in Besov spaces without smallness are included.
\end{remark}
\begin{remark}
All estimates provided here have their dual equivalent. This is explained in detail in Appendix~\ref{sec:AA}. This means in particular that the multiplier conditions above are bound to commute with duality. The related Bilaplacian~\eqref{eq:bilaplace} provides a clear picture in two space dimensions. Let us assume that the condition $C_{s,p}$ is the necessary boundary regularity for the Poisson equation with zero boundary values in 2D. Please note that by duality for the Laplace equation $C_{2,p}=C_{0,p'}$. Hence, in the case $s=1$, we find that~\eqref{eq:bilaplace} implies, that the boundary regularity has to be in $C_{2,p}\cap C_{0,p}=C_{2,p}\cap C_{2,p'}=C_{0,p'}\cap C_{2,p}$. But that reflects precisely the duality of the Stokes equation with perfect slip boundary conditions.
\end{remark}
\begin{proof}[Proof of Theorem~\ref{thm:stokessteadyF}]
Similarly to Theorem \ref{thm:StokesH} we can reduce the problem to the case $\mathfrak g=h=0$ with the help of Theorem \ref{thm:laplace}.
 Furthermore, let us suppose that $\bfu$ and $\pi$ are sufficiently smooth. We will remove this restriction at the end of the proof by some approximation procedure which will ensure existence of a solution. Uniqueness is obvious if $p\geq2$, while the case $p<2$ can eventually be treated by duality (see Appendix~\ref{sec:AA}).

By assumption there is $\ell\in\mathbb N$ and Lipschitz functions $\varphi_1,\dots,\varphi_\ell\in\mathscr M^{s+1-1/p,p}(\mathbb R^{n-1})(\delta)$, where $\delta$ is sufficiently small, satisfying \ref{A1}--\ref{A3}.
We clearly find an open set $\mathcal U^0\Subset\Omega$ such that $\Omega\subset \cup_{j=0}^\ell \mathcal U^j$. Finally, we consider a decomposition of unity $(\xi_j)_{j=0}^\ell$ with respect to the covering
$\mathcal U^0,\dots,\mathcal U^\ell$ of $\Omega$. 
For $j\in\{1,\dots,\ell\}$ we consider the extension $\bfPhi_j$ of $\varphi_j$ given by \eqref{eq:Phi} with inverse $\bfPsi_j$.
 
We transform now the system \eqref{eq:Stokesdiv} to a corresponding problem on the half space. This is reminiscent of \cite[Section 3]{Br2} but more care is required for the boundary conditions.
 Let us fix $j\in\{1,\dots,\ell\}$ and assume, without loss of generality, that the reference point $y_j=0$ and that the outer normal at~$0$ is pointing in the negative $x_n$-direction (this saves us some notation regarding the translation and rotation of the coordinate system).
We multiply $\bfu$ by $\xi_j$ and obtain for $\bfu_j:=\xi_j\bfu$, $\Pi_j:=\xi_j\pi$ and $\bfF_j:=\xi_j\bfF$ the equation
\begin{align}\label{eq:Stokes2}
\Delta \bfu_j-\nabla\Pi_j=[\Delta,\xi_j]\bfu-[\nabla,\xi_j]\Pi+[\Div,\xi_j]\bfF-\Div\bfF_j,\quad\Div\bfu_j=\nabla\xi_j\cdot\bfu,\quad\text{in}\quad\Omega,\\
 \label{eq:Stokes2b}
\bfu_j\cdot\bfn=0,\quad (\nabla\bfu_j\,\bfn+\bfF_j\,\bfn)_{\bftau}=-\alpha (\bfu_j)_\bftau-(\bfu\otimes\nabla\xi_j\,\bfn)_{\bftau}+\xi_j\GG,\quad\text{on}\quad\partial\Omega,
\end{align}
with the commutators $[\Delta,\xi_j]=\Delta\xi_j+2\nabla\xi_j\cdot\nabla$, $[\nabla,\xi_j]=\nabla\xi_j$ and $[\Div,\xi_j]=\nabla\xi_j\cdot$.
Finally, we set $\bfv_j:=\bfu_j\circ\bfPhi_j$, $\theta_j:=\Pi_j\circ\bfPhi_j$, 
\footnote{We denote by $\Delta^{-1}_{\mathbb H}$ the solution operator to the Laplace equation in $\mathbb H$ with homogeneous boundary conditions on $\partial\mathbb H$.}
\begin{align*}
\bfG_j&:=-\nabla\Delta^{-1}_{\mathbb H}\Big(\mathrm{det}(\nabla\bfPhi_j)(\nabla\xi_j\cdot\bfu)\circ\bfPhi_j([\Delta,\xi_j]\bfu-[\nabla,\xi_j]\Pi+[\Div,\xi_j]\bfF)\circ\bfPhi_j\Big),\\\bfH_j&:=\mathbf{B}_j\bfF_j\circ\bfPhi_j,
\end{align*}
$h_j=\mathrm{det}(\nabla\bfPhi_j)(\nabla\xi_j\cdot\bfu)\circ\bfPhi_j$, $\alpha_j:=\alpha\circ\bfPhi_j$ as well as $\boldsymbol{\mathfrak G}_j:=\xi_j\GG\circ\bfPhi_j$ and 
obtain the equations
\begin{align}\label{eq:Stokes3}
&\Div\big(\bfA_j\nabla\bfv_j)-\Div(\mathbf{B}_j\theta_j)=-\Div(\bfG_j+\bfH_j),
\quad\mathbf{B}_j^\top:\nabla\bfv_j=h_j,\quad\text{in}\quad\mathbb H,\\
&\bfv_j\cdot\nabla\varphi_j^\perp=0,\quad
(\widetilde{\bfA_j\nabla\bfv_j}+\tilde\bfG_j+\tilde \bfH_j)^n=-\big(\bfv_j\otimes\nabla\xi_j\circ\bfPhi_j\,\nabla\varphi_j^\perp)
 \label{eq:Stokes3b}\\&\qquad\qquad\qquad\qquad\qquad\qquad\qquad\qquad\qquad+\big(\big(\bfv_j\otimes\nabla\xi_j\circ\bfPhi_j\,\nabla\varphi_j^\perp\big)\cdot\nabla\varphi_j^\perp\big)\nabla\varphi_j^\perp\big)\nonumber\\
&\qquad\qquad\qquad\qquad\qquad\qquad\qquad\qquad\qquad-\alpha_j\,\big(\bfv_j-(\bfv_j\cdot \nabla\varphi_j^\perp)\nabla\varphi_j^\perp\big)+\boldsymbol{\mathfrak G}_j,\quad\text{on}\quad\partial\mathbb H,\nonumber
\end{align}
where $\bfA_j:=\mathrm{det}(\nabla\bfPhi_j)\nabla\bfPsi_j^\top\circ\bfPhi_j\nabla\bfPsi_j\circ\bfPhi_j$ and $\mathbf{B}_j:=\mathrm{det}(\nabla\bfPhi_j)\nabla\bfPsi_j\circ\bfPhi_j$
 (note that we have $\Div\mathbf{B}_j=0$ due to the Piola identity).
 This can be rewritten as (note that $((\mathbf{B}_j-\mathbb I_{n\times n})\bfe_ n)^i=0$ for $i=1,\dots,n-1$)
\begin{align}\label{eq:Stokes4}
\Delta\bfv_j-\nabla\theta_j&=\Div\big((\mathbb I_{n\times n}-\bfA_j)\nabla\bfv_j)+\Div((\mathbf{B}_j-\mathbb I_{n\times n})\theta_j)-\Div(\bfG_j+\bfH_j)\quad\text{in}\quad\mathbb H,\\
&\Div\bfv_j=(\mathbb I_{n\times n}-\mathbf{B}_j)^\top:\nabla\bfv_j+h_j\quad\text{in}\quad\mathbb H,\label{eq:Stokes4a}\\
 \label{eq:Stokes4b}
\bfv_j\cdot\bfe_n&=\mathfrak g(\bfv_j),\quad \partial_n\tilde\bfv_j+(\widetilde{(\bfA_j-\mathbb I_{n\times n})\nabla\bfv_j}+\tilde\bfG_j+\tilde \bfH_j)=\boldsymbol{\mathfrak G}(\bfv_j),\quad\text{on}\quad\partial\mathbb H,
\end{align}
where 
\begin{align*}
\mathfrak g(\bfv_j)&=\bfv_j\cdot\bfe_n-\bfv_j\cdot\nabla\varphi_j^\perp,\\
\boldsymbol{\mathfrak G}(\bfv_j)&=-\big(\bfv_j\otimes\nabla\xi_j\circ\bfPhi_j\,\nabla\varphi_j^\perp
-\big(\big(\bfv_j\otimes\nabla\xi_j\circ\bfPhi_j\,\nabla\varphi_j^\perp\big)\cdot\nabla\varphi_j^\perp\big)\nabla\varphi_j^\perp\big)\\
&-\alpha_j\,\big(\bfv_j-(\bfv_j\cdot \nabla\varphi_j^\perp)\nabla\varphi_j^\perp\big)+\GG_j\\
&=:\boldsymbol{\mathfrak G}^1(\bfv_j)+\boldsymbol{\mathfrak G}^2(\bfv_j)+
\boldsymbol{\mathfrak G}_j.
\end{align*}
Setting 
\begin{align*}
\bfS (\bfv,\theta)&=-\bfS _1(\bfv)-\bfS_2 (\theta),\\
\bfS_1 (\bfv)&=(\mathbb I_{n\times n}-\bfA_j)\nabla\bfv,\\
\bfS_2 (\theta)&=(\mathbf{B}_j-\mathbb I_{n\times n})\theta,\\
\mathfrak s(\bfv)&=(\mathbb I_{n\times n}-\mathbf{B}_j)^\top:\nabla\bfv,
\end{align*} we can finally write
\eqref{eq:Stokes3} as
\begin{align}\label{eq:Stokes5}
\Delta\bfv_j-\nabla\theta_j&=-\Div\big(\bfS(\bfv_j,\theta_j)+\bfG_j+\bfF_j\big),
\quad\Div\bfv_j=\mathfrak s(\bfv_j)+h_j\quad\text{in}\quad\mathbb H,\\
 \label{eq:Stokes5b}
\bfv_j\cdot\bfe_n&=\mathfrak g(\bfv_j),\quad \partial_n\tilde\bfv_j+\tilde \bfS^n(\bfv_j,\theta_j)+\tilde\bfG_j^n+\tilde\bfF_j^n=\boldsymbol{\mathfrak G}(\bfv_j),\quad\text{on}\quad\partial\mathbb H.
\end{align}
We can now apply the estimates for the half space from Theorem \ref{thm:StokesH} obtaining
\begin{align}\label{eq:0201cF}
\begin{aligned}
\|\bfv_j\|_{W^{s,p}(\mathbb H)}+\|\theta_j\|_{W^{s-1,p}(\mathbb H)}&\lesssim \|\bfS (\bfv_j,\theta_j)+\bfH_j+\bfG_j\|_{W^{s-1,p}(\mathbb H)}+\|\mathfrak s(\bfv_j)+h_j\|_{W^{s-1,p}(\mathbb H)}\\
&+\|\mathfrak g(\bfv_j)\|_{W^{s-1/p,p}(\partial\mathbb H)}+\|\boldsymbol{\mathfrak G}(\bfv_j)\|_{W^{s-1-1/p,p}(\partial\mathbb H)}.
\end{aligned}
\end{align}
Our remaining task consists in estimating the right-hand side. As in \cite[Section 3]{Br} we have
\begin{align*}
\|\bfS(\bfv_j,\theta_j)\|_{W^{s-1,p}(\mathbb H)}+\|\mathfrak s(\bfv_j)\|_{W^{s-1,p}(\partial\mathbb H)}&\lesssim \|\bfPhi_j\|_{\mathscr M^{s,p}(\mathbb H)}\Big(\|\bfv\|_{W^{s,p}(\mathbb H)}+\|\theta\|_{W^{s-1,p}(\mathbb H)}\Big)
\end{align*}
Clearly, it holds\footnote{Note the inclusions between Sobolev multiplier spaces from \cite[Corollary 4.3.2]{MaSh}.}
\begin{align*}
 \|\bfPhi_j\|_{\mathscr M^{s,p}(\mathbb H)}\lesssim  \|\bfPhi_j\|_{\mathscr M^{s+1,p}(\mathbb H)}\lesssim \|\varphi_j\|_{\mathscr M^{s+1-1/p,p}(\partial\mathbb H)}
\end{align*}
using \eqref{eq:MS}.

We proceed by 

\begin{align*}
\|\bfH_j\|_{W^{s-1,p}(\mathbb H)}&\lesssim \|\mathbf{B}_j\|_{\mathscr M^{s-1,p}_{\tt or}(\mathbb H)}\|\bfF_j\circ\bfPhi_j\|_{W^{s-1,p}(\mathbb H)}\\
&\lesssim \|\nabla\bfPsi_j^\top\circ\bfPhi_j\|_{\mathscr M^{s-1,p}_{\tt or}}
\|\bfF_j\|_{W^{s-1,p}(\mathbb H)}\lesssim \|\bfF\|_{W^{s-1,p}(\Omega)},
\end{align*}
where the hidden constant depends on $\mathrm{det}(\nabla\bfPhi_j)$ and  $\|\bfPhi_j\|_{\mathscr M^{s,p}(\mathbb H)}$ being controlled by \eqref{eq:detJ}.
The continuity of $\nabla\Delta_{\mathbb H}^{-1}$ yields
 \begin{align*}
 \|\bfG_j\|_{W^{s-1,p}(\mathbb H)}
&\lesssim\|\bfu\|_{W^{s-1,p}(\Omega)}+\|\pi\|_{W^{s-2,p}(\Omega)}+ \|\bfF\|_{W^{s,p}(\Omega)},
 \end{align*}
cf.~ \eqref{lem:9.4.1}. 
%
From now on we must suppose that $W^{s,p}(\Omega)\hookrightarrow W^{1,2}(\Omega)$. Choosing $s_0\in\R$ such that $W^{1,2}(\Omega)\hookrightarrow W^{s_0,p}(\Omega)$ with $s_0<s$, there is $\gamma\in(0,1)$ such that
 \begin{align}\nonumber
 \|\bfu\|_{W^{s-1,p} (\Omega)}&\leq \|\bfu\|_{W^{s,p} }^{\gamma}\|\bfu\|_{W^{s_0,p} (\Omega)}^{1-\gamma}\lesssim \| \bfu\|_{W^{s,p}(\Omega) }^{\alpha}\| \bfu\|_{W^{1,2}(\Omega) }^{1-\gamma} \\&\lesssim\| \bfu\|_{W^{s,p}(\Omega) }^{\gamma}\big(\|\bfF\|_{L^{2} }+\|\GG\|_{W^{-1/2,2}(\partial\Omega)}\big)^{1-\alpha}\nonumber\\&\lesssim\|\bfu\|_{W^{s,p} }^{\gamma}\big(\|\bfF\|_{W^{s-1,p} (\Omega)}+\|\GG\|_{W^{s-1-1/p,p} (\partial\Omega)}\big)^{1-\gamma}\nonumber\\
&\leq\delta \|\bfu\|_{W^{s,p} (\Omega)}+c(\delta)\big(\|\bfF\|_{W^{s-1,p}(\Omega) }+\|\GG\|_{W^{s-1-1/p,p} (\partial\Omega)}\big)\label{eq:asin}
 \end{align}
for $\delta>0$ arbitrary. Note that we used the energy estimate for a weak solution in $W^{1,2} $ in the above. Similarly, we obtain
 \begin{align}
 \|\pi\|_{W^{s-2,p} (\Omega)}
&\leq\delta \|\pi\|_{W^{s-1,p}(\Omega) }+c(\delta)\big(\|\bfF\|_{W^{s-1,p}(\Omega) }+\|\GG\|_{W^{s-1-1/p,p} (\partial\Omega)}\big)\label{eq:asinpi}
 \end{align}

Thus the main task consists in estimating the last line in \eqref{eq:0201cF}.\footnote{This actually is the only place, where more regularity on the boundary is required compared to the case of no-slip boundary conditions considered in \cite[Section 3.2]{Br}.}
We have by \eqref{J}, \eqref{eq:MS} and \eqref{eq:SoMo'}
\begin{align*}
\|\mathfrak g(\bfv_j)\|_{W^{s-1/p,p}(\partial\mathbb H)}&\leq\,\|\bfe_n-\nabla\varphi_j^\perp\|_{\mathscr M^{s-1/p,p}_{\tt or}(\partial\mathbb H)}\|\bfv_j\|_{W^{s-1/p,p}(\partial\mathbb H)}\\
&\lesssim\|\nabla\varphi_j\|_{\mathscr M_{\tt or}^{s-1/p,p}(\partial\mathbb H)}\|\bfv_j\|_{W^{s,p}(\mathbb H)}\\
&\lesssim\|\varphi_j\|_{\mathscr M^{s+1-1/p,p}(\partial\mathbb H)}\|\bfv_j\|_{W^{s,p}(\mathbb H)}.
\end{align*}
Similarly, it holds 
\begin{align*}
\|\boldsymbol{\mathfrak G}^2(\bfv_j)\|_{W^{s-1-1/p,p}(\partial\mathbb H)}
\leq\|\bfv_j-(\bfv_j\cdot \nabla\varphi_j^\perp)\nabla\varphi_j^\perp\|_{W^{s-1/p,p}(\partial\mathbb H)}
&\lesssim\|\varphi_j\|_{\mathscr M^{s+1-1/p,p}(\partial\mathbb H)}\|\bfv_j\|_{W^{s,p}(\mathbb H)}
\end{align*}
using that $\alpha\in \mathscr M^{s-1-1/p,p}_{s-1/p,p}$ (which implies the same for $\alpha_j$ by \eqref{lem:9.4.1}).
Furthermore, we obtain
\begin{align*}
&\|\boldsymbol{\mathfrak G}^1(\bfv_j)\|_{W^{s-1-1/p,p}(\partial\mathbb H)}
\\&\qquad\leq
\Big(\|\bfv_j\otimes\nabla\xi_j\circ\bfPhi_j\|_{W^{s-1-1/p,p}(\partial\mathbb H)}+\|\bfv_j\otimes\nabla\xi_j\circ\bfPhi_j
\nabla\varphi_j^\perp\cdot\nabla\varphi_j^\perp\|_{W^{s-1-1/p,p}(\partial\mathbb H)}\Big)\\
&\qquad\lesssim \big(1+\|\nabla\varphi_j\|_{\mathscr M^{s-1-1/p,p}_{\tt or}(\partial\mathbb H)}^2\big)\|\bfv_j\|_{W^{s-1-1/p,p}(\partial\mathbb H)}\\
&\qquad\lesssim\|\bfu\|_{W^{s-1-1/p,p}(\mathbb H)}\lesssim\|\bfu\|_{W^{s-1,p} }
\end{align*}
 using \eqref{lem:9.4.1} in the penultimate and $\Div\bfu=0$ in the ultimate step.
The last term can be estimated as in \eqref{eq:asin} (which requires $p\geq2$).
Hence we obtain
  \begin{align}\label{almost1}
\|\boldsymbol{\mathfrak G}^1(\bfv_j)\|_{W^{s-1-1/p,p}(\partial\mathbb H)}
&\leq\kappa\|\bfu\|_{W^{s,p} }+c(\kappa)\|\bfF\|_{W^{s-1,p} }
 \end{align}
for any $\kappa>0$.
Moreover, we have
\begin{align*}
\|\boldsymbol{\mathfrak G}_j\|_{W^{s-1-1/p,p}(\partial\mathbb H)}\lesssim 
\|\boldsymbol{\mathfrak G}\|_{W^{s-1-1/p,p}(\partial\Omega)}
\end{align*}
by \eqref{lem:9.4.1}.

 If the assumptions from part (a) are in force we clearly have$^4$
$$\|\varphi_j\|_{\mathscr M^{s+1-1/p,p}_{\tt or}(\partial\mathbb H)}\lesssim\|\varphi_j\|_{\mathscr M^{s+1-1/p,p}_{\tt or}(\partial\mathbb H)}\leq\,c\delta,$$
while we can estimate
\begin{align*}
\|\varphi_j\|_{\mathscr M^{s+1-1/p,p}(\partial\mathbb H)}=\|\nabla\varphi_j\|_{\mathscr M^{s-1/p,p}_{\tt or}(\partial\mathbb H)}\lesssim r^{s-n/p}\|\nabla\varphi_j\|_{W^{s-1/p,p}(\partial\mathbb H)}
\end{align*}
under the conditions of (b) using \eqref{eq:MSc}. Here we assume that $\varphi_j$ is supported in a ball of radius $r$ and choose $r$ sufficiently small to come to the same conclusion as for (a).
Putting everything together shows for all $j\in\{1,\dots,\ell\}$
 \begin{align*}
\begin{aligned}
\|\nabla\bfv_j\|_{W^{s-1,p}(\mathbb H)}+\|\theta_j\|_{W^{s-1,p}(\mathbb H)}&\lesssim \delta
 \|\bfu\|_{W^{s,p}(\Omega)}+ \delta \|\pi\|_{W^{s-1,p}(\Omega)}\\&+\|\bfF\|_{W^{s-1,p}(\Omega)}+\|\boldsymbol{\mathfrak G}\|_{W^{s-1-1/p,p}(\partial\Omega)}.
\end{aligned}
\end{align*}
Clearly, the same estimate (even without the first two terms on the right-hand side) holds for $j=0$ by local regularity theory for the Stokes system.
Summing over $j=0,1,\dots,\ell$ and choosing
 $\delta$ small enough proves the claimed estimate
provided $\bfu$ and $\pi$ are sufficiently smooth.
 Let us finally remove this assumption which is not a priori given.
Applying a standard regularisation procedure (by convolution with mollifying kernel) to the functions $\varphi_1,\dots,\varphi_\ell$ from \ref{A1}--\ref{A3} in the parametrisation of $\partial{\Omega}$ we obtain a smooth boundary. Classically, the solution to the corresponding Stokes system is smooth.
 Such a procedure is standard and has been applied, for instance, in \cite[Section 4]{CiMa}. It is possible to do this in a way that the original domain is included in the regularised domain to which we extend the function $\bff$ by means of an extension operator.
 The regularisation applied to the $\varphi_j's$ converges on all Besov spaces with $p<\infty$. It does not converge on $W^{1,\infty}(\R^{n-1})$, but the regularisation does not expand the $W^{1,\infty}(\R^{n-1})$-norm, which is sufficient. Following the arguments above we obtain
 \eqref{eq:main} for the regularised problem with a uniform constant. The limit passage is straightforward since \eqref{eq:Stokes} is linear.
\end{proof}
We now consider the case where the embedding $W^{s,p}(\Omega)\hookrightarrow W^{1,2}(\Omega)$  fails.
\begin{corollary}\label{cor:stokessteadyF}
Let $n=2,3$, $p\in(1,2)$ and $s\geq1$. Assume that
$\alpha:\partial\Omega\rightarrow[0,\infty)$ is a measurable function belonging to the class $\alpha\in \mathscr M_{s-1/p,p}^{s-1-1/p,p}\cap \mathscr M_{s-1/p',p'}^{s-1-1/p',p'}$.
Let $\Omega$ be a bounded Lipschitz domain and suppose
that either
\begin{enumerate}
\item[(a)]  $p(s-1)\leq n$ and $\partial\Omega$ belongs to the class
$\mathscr M^{s+1-1/p,p}\cap \mathscr M^{s+1-1/p',p'}(\delta)$ for some sufficiently small $\delta>0$ or
\item[(b)] $p(s-1)> n$ and $\partial\Omega$ belongs to the class
$W^{s+1-1/p',p'}$. 
\end{enumerate}
Then the estimate from Theorem \ref{thm:stokessteadyF} holds.
\end{corollary}
\begin{proof}
Let us first consider the case $s=1$ and $p\in(1,2)$. In this case we can apply a straight forward duality argument (see Appendix~\ref{sec:AA}) to obtain the desired estimate. With the $W^{1,p}$-estimate at hand we repeat the proof from Theorem \ref{thm:stokessteadyF}. They only term which must be treated differently is the lower order term corresponding to \eqref{eq:asin} above. We have now
\begin{align*}
\|\bfu\|_{W^{s-1,p}(\Omega)}&+\|\pi\|_{W^{s-2,p}(\Omega)}\\&\leq\,\delta\big( \|\bfu\|_{W^{s,p}(\Omega)}+\|\pi\|_{W^{s-1,p}(\Omega)}\big)+c(\delta)\big( \|\nabla\bfu\|_{L^{p}(\Omega)}+\|\pi\|_{L^{p}(\Omega)}\big)\\
&\leq\,\delta\big( \|\bfu\|_{W^{s,p}(\Omega)}+\|\pi\|_{W^{s-1,p}(\Omega)}\big)+c(\delta) \big(\|\bfF\|_{L^{p}(\Omega)}+\|\boldsymbol{\mathfrak G}\|_{W^{-1/p,p}(\partial\Omega)}\big)\\
&\leq\,\delta\big( \|\bfu\|_{W^{s,p}(\Omega)}+\|\pi\|_{W^{s-1,p}(\Omega)}\big)+c(\delta) \big(\|\bfF\|_{W^{s-1,p}(\Omega)}+\|\boldsymbol{\mathfrak G}\|_{W^{s-1-1/p,p}(\partial\Omega)}\big)
\end{align*}
and the proof can be completed as before.
\end{proof}

We are now able to deduce an estimate in spaces $W^{\sigma,p} $ with $\sigma\in(-\infty,1)$ by a duality argument, see Appendix \ref{sec:AA} for details. Considering first the case of homogeneous boundary data 
 we obtain
$$\|\nabla\bfu\|_{W^{\sigma-1,p}_n(\Omega)}
\lesssim\|\bff\|_{(W^{2-\sigma,p'}_n(\Omega))'}$$
for a solution $\bfu$ to \eqref{eq:Stokesdiv}, where
$\bff(\bfphi):=-\langle\bfF,\nabla\bfphi\rangle$ for smooth $\bfphi$, with $\bfphi\cdot n=0$.
As in the preamble of Subsection~\ref{sec:half} this is without loss of generality as solving the corresponding Neumann problems (see Corollary \ref{cor:laplace}) one can also consider non-trivial boundary conditions (see Appendix~\ref{sec:AB}).
For that propose we need the trace space
\begin{align*}
\mathscr W^{s,p}(\partial\Omega)&:=\{\nabla v\cdot\bfn|_{\partial\Omega}:\,v\,\in W^{s,p}_n(\Omega)\},\\
\|{\chi}\|_{\mathscr W^{s-1-1/p,p}(\partial\Omega)}&:=\inf\{\|v\|_{W^{s,p}(\Omega)}:\,\,,\,v\in W^{s,p}_n(\Omega),\,\,\nabla v\cdot\bfn={\chi}\},
\end{align*}
where we understand the trace $\nabla v\cdot\bfn|_{\partial\Omega}$ and the equality  in the sense of distributions.
 and conclude with the following corollary.
\begin{corollary}\label{cor:stokessteadyF}
Let $n=2,3$, $p\in(1,\infty)$ and $\sigma<1$. Set $s:=2-\sigma$ and assume that
$\alpha:\partial\Omega\rightarrow[0,\infty)$ is a measurable function belonging to the class $\alpha\in \mathscr M_{s-1/p,p}^{s-1-1/p,p}$.
Let $\Omega$ be a bounded Lipschitz domain and suppose
that either
\begin{enumerate}
\item[(a)]  $p(s-1)\leq n$ and $\partial\Omega$ belongs to the class
$\mathscr M^{s+1-1/p,p}(\delta)$ for some sufficiently small $\delta>0$ or
\item[(b)] $p(s-1)> n$ and $\partial\Omega$ belongs to the class
$W^{s+1-1/p,p}$. 
\end{enumerate}
For any $\bff\in (W^{2-\sigma,p'}_n({\Omega}))'$ , $h\in W^{\sigma-1,p}(\Omega)$, $\mathfrak g\in \mathscr W^{\sigma-1/p,p}(\partial\Omega)$ with $\int_{\Omega}h\dx=\int_{\partial\Omega}\mathfrak g\,\dd\mathcal H^{n-1}$ and $\GG\in \mathscr W^{\sigma-1-1/p,p}(\partial\Omega)$
there is a unique solution to \eqref{eq:Stokesdiv} and we have
\begin{align*}
\begin{aligned}
\|\nabla\bfu\|_{W^{\sigma-1,p}_n(\Omega)}
&\lesssim\|\bff\|_{(W^{2-\sigma,p'}(\Omega))'}+\|h\|_{W^{\sigma-1,p}_n(\Omega)}\\&+\|\mathfrak g\|_{\mathscr W^{\sigma-1/p,p}(\partial\Omega)}+\|\GG\|_{\mathscr W^{\sigma-1-1/p,p}(\Omega)}.
\end{aligned}
\end{align*}
\end{corollary}

\section{The problem in non-divergence form}\label{sec:stokessteady}
In this section we consider the steady Stokes system
\begin{align}\label{eq:Stokes}
\begin{aligned}
\Delta \bfu-\nabla\pi=-\bff,\quad\Div\bfu=h,\quad\text{in}\quad\Omega,\\
\bfu\cdot\bfn=\mathfrak g,\quad (\nabla\bfu\,\bfn)_{\bftau}+\alpha \bfu_\bftau=\GG,\quad\text{on}\quad\partial\Omega,
\end{aligned}
\end{align}
in a domain $\Omega\subset\R^n$ with unit normal $\bfn$, where $\bfv_\bftau:=\bfv-(\bfv\cdot\bfn)\bfn$. The result given in the following theorem is a maximal regularity estimate for the solution in terms of the right-hand side and the boundary datum under minimal assumptions on the regularity of $\Omega$ (the corresponding multiplier spaces are introduced in Section \ref{sec:SM}). We start with the case $p\geq \frac{2n}{n+1}$, see Corollary~\ref{cor:stokessteady} for the missing one.
\begin{theorem}\label{thm:stokessteady}
Let $n=2,3$, $p\in(\frac{2n}{n+1},\infty)$ and $s\in[2,\infty)$.
Assume that
$\alpha:\partial\Omega\rightarrow[0,\infty)$ is a measurable function belonging to the class $\alpha\in \mathscr M_{s-1/p,p}^{s-1-1/p,p}$.
Let $\Omega$ be a bounded Lipschitz domain and suppose that either
\begin{enumerate}
\item[(a)]  $p(s-1)\leq n$ and $\partial\Omega$ belongs to the class
$\mathscr M^{s+1-1/p,p}(\delta)$ for some sufficiently small $\delta>0$ or
\item[(b)] $p(s-1)> n$ and $\partial\Omega$ belongs to the class
$W^{s+1-1/p,p}$. 
\end{enumerate}
For any $\bff\in W^{s-2,p}({\Omega})$, $h\in W^{s-1,p}(\Omega)$ with $\int_{\Omega}h\dx=0$, $\mathfrak g\in W^{s-1/p,p}(\partial\Omega)$ and $\GG\in W^{s-1-1/p,p}(\partial\Omega)$
there is a unique solution $(\bfu,\pi)$ to \eqref{eq:Stokes} and we have
\begin{align}\label{eq:main}
\|\bfu\|_{W^{s,p}({\Omega})}+\|\pi\|_{W^{s-1,p}({\Omega})}\lesssim\|\bff\|_{W^{s-2,p}({\Omega})}+\|\mathfrak g\|_{W^{s-1/p,p}(\partial\Omega)}+\|\GG\|_{W^{s-1-1/p,p}(\partial\Omega)}.
\end{align}
The constant in \eqref{eq:main} depends on the
$\mathscr M^{s+1-1/p}$- or $W^{s+1-1/p}$-norms
of the local charts in the parametrisation of  $\partial\Omega$.
\end{theorem}
\begin{remark}\label{rem:main}

We refer to Remark \ref{rem:main2} for details regarding the assumption on $\partial\Omega$ formulated in terms of Besov spaces. We just mention here the case
$s=p=2$ and $n=3$. Then Theorem \ref{thm:stokessteady} (b) applies and we require that 
$\partial\Omega\in W^{5/2,2}$.
\end{remark}
\begin{remark}
If $n/2<p$ one easily checks that $W^{1-1/p,p}\hookrightarrow \mathscr M^{1-1/p,p}_{2-1/p,p}$ such that Theorem Theorem \ref{thm:stokessteady} applies provided $\alpha\in W^{1-1/p,p}(\partial\Omega)$.
\end{remark}
\begin{proof}

As in the proof of Theorem \ref{thm:stokessteadyF}
we assume $\mathfrak g=h=0$ (using now Theorem \ref{thm:laplaceF})
and suppose that all quantities are sufficiently smooth.
Also
we introduce again local coordinates and transform the system (see also \cite[Section 4]{Br2}). With the same notation as there, setting also $\mathcal S(\bfv,\theta)=\Div\bfS(\bfv,\theta)$ and
\begin{align*}
\bfg_j:=\mathrm{det}(\nabla\bfPhi_j)(\nabla\xi_j\cdot\bfu)\circ\bfPhi_j([\Delta,\xi_j]\bfu-[\nabla,\xi_j]\Pi+\xi_j\bff)\circ\bfPhi_j
\end{align*}
we obtain
\begin{align}\label{eq:Stokes3'F}
\Delta\bfv_j-\nabla\theta_j&=\mathcal S (\bfv_j,\theta_j)+\bfg_j,
\quad\Div\bfv_j=\mathfrak s(\bfv_j)+h_j,\quad\text{in}\quad\mathbb H,\\
 \label{eq:Stokes5bF}
\bfv_j\cdot\bfe_n&=\mathfrak g(\bfv_j),\quad \partial_n\tilde\bfv_j =\boldsymbol{\mathfrak G}(\bfv_j),\quad\text{on}\quad\partial\mathbb H.
\end{align}

The estimate for the half space (see Theorem \ref{thm:StokesH})
implies
\begin{align}\label{eq:0201c}
\begin{aligned}
\|\bfv_j\|_{W^{s,p} }+\|\theta_j\|_{W^{s-1,p} }&\lesssim \|\mathcal S (\bfv_j,\theta_j)+\bfg_j\|_{W^{s-2} }+\|\mathfrak s(\bfv_j)+h_j\|_{W^{s-1,p} }\\
&+\|\mathfrak g(\bfv_j)\|_{W^{s-1/p,p} }+\|\boldsymbol{\mathfrak G}(\bfv_j)\|_{W^{s-1-1/p,p} }
\end{aligned}
\end{align}
for all $s\geq 1$. Except for the first on the right-hand side all terms can be estimated exactly as in the proof of Theorem \ref{thm:stokessteadyF}. We clearly have,
\begin{align*}
 \|\mathcal S (\bfv_j,\theta_j)\|_{W^{s-2} }\leq  \|\bfS (\bfv_j,\theta_j)\|_{W^{s-1} }
\end{align*}
such that also this estimate is accordingly. Finally, we have
 \begin{align*}
 \|\bfg_j\|_{W^{s-2,p} }+ \|h_j\|_{W^{s-1,p} }&\lesssim \|\bfu\circ\bfPhi_j\|_{W^{s-1,p} }+\|\pi\circ\bfPhi_j\|_{W^{s-2,p} }+ \|\bff\circ\bfPhi_j\|_{W^{s-2} }\\
& \lesssim \|\bfu\|_{W^{s-1,p} }+\|\pi\|_{W^{s-2,p} }+\|\bff\|_{W^{s-2,p} },
 \end{align*}
 
%
In order to estimate the lower order terms on the right-hand side we set $\bfF:=\nabla\Delta_\Omega^{-1}\bff$ and employ the estimate from Theorem \ref{thm:stokessteadyF} yielding for $\delta>0$ arbitrary
\begin{align*}
\|\bfu\|_{W^{s-1,p}(\Omega)}+\|\pi\|_{W^{s-2,p}(\Omega)}&\leq\,\delta\big(\|\bfu\|_{W^{s,p}(\Omega)}+\|\pi\|_{W^{s-1,p}(\Omega)}\big)+c(\delta)\big(\|\nabla \bfu\|_{L^{p}(\Omega)}+\|\pi\|_{L^{p}(\Omega)}\big)\\
&\leq\,\delta \|\bfu\|_{W^{s,p}(\Omega)}+c(\delta) \big(\|\bfF\|_{L^{p}(\Omega)}+\|\GG\|_{W^{-1/p,p}({\partial\Omega})}\big)\\
&\leq\,\delta\|\bfu\|_{W^{s,p}(\Omega)}+c(\delta)\big( \|\bff\|_{W^{s-2,p}(\Omega)}+\|\GG\|_{W^{s-1-1/p,p}({\partial\Omega})}\big)
\end{align*}
and the proof can be completed as before. 
\end{proof}

Similarly to Corollary \ref{cor:stokessteadyF} we can treat the missing indices with an additional assumption on the multipliers.
\begin{corollary}\label{cor:stokessteady}
Let $n=2,3$, $p\in(1,\frac{2n}{n+1})$ and $s\geq2$. Assume that
$\alpha:\partial\Omega\rightarrow[0,\infty)$ is a measurable function belonging to the class $\alpha\in \mathscr M_{s-1/p,p}^{s-1-1/p,p}\cap \mathscr M_{s-1-1/p',p'}^{s-2-1/p',p'}$.
Let $\Omega$ be a bounded Lipschitz domain and suppose
that either
\begin{enumerate}
\item[(a)]  $p(s-1)\leq n$ and $\partial\Omega$ belongs to the class
$\mathscr M^{s+1-1/p,p}\cap \mathscr M^{s-1/p',p'}(\delta)$ for some sufficiently small $\delta>0$ or
\item[(b)] $p(s-1)> n$ and $\partial\Omega$ belongs to the class
$W^{s+1-1/p,p}\cap W^{s-1p',p'}$. 
\end{enumerate}
Then the estimate from Theorem \ref{thm:stokessteady} holds.
\end{corollary}
\appendix 
\section{Duality}
\label{sec:AA}
Here we present some duality arguments which show how one deduce estimates in a certain dual space from the corresponding pre-dual theory for the case of a bounded domain $\Omega \subset\R^n$.
\subsection{Duality from $L^{p'}$ to $L^p$}
\label{ssec:dualp}
The centre of the duality is related to the space with the gradient $\nabla \bfu\in L^2(\Omega)$. First we show how to turn $L^p(\Omega)$-estimates to $L^{p'}(\Omega)$-estimates for the gradient. Due to the results for the Neumann problem, see Appendix \ref{sec:AB} below, we can argue as in the preamble of Subsection~\ref{sec:half} and assume that \[
\bfu\cdot \bfn=0\text{ on }\partial \Omega\text{ and }\divergence\bfu=0\text{ in }\Omega.
\]
Further we sub summarize $f,F,\GG$ as a general right hand side $\tilde{f}$, namely
\[
\int_\Omega\nabla\bfu:\nabla\bfphi\dx+\int_{\partial\Omega}\alpha\bfu\cdot\bfphi\,\dd\mathcal H^{n-1} = \tilde f(\bfphi):=\int_{\partial\Omega}\GG\cdot\bfphi\,\dd\mathcal H^{n-1}-\int_{\Omega}\bfF:\nabla\bfphi\dx
\]
for all $\bfphi\in W^{1,2}_{n,\Div}(\Omega)$.
By the dual norm formula we have
\begin{align*}
\norm{\nabla \bfu}_{L^p(\Omega)}=\sup_{\bfPsi\in C^\infty_c(\Omega), \norm{\bfPsi}_{L^{p'}(\Omega)}\leq 1} \int_\Omega\nabla \bfu:\bfPsi\dx.
\end{align*}
For a fixed $\bfPsi\in C^\infty_c(\Omega)$ with $\norm{\bfPsi}_{L^{p'}}\leq 1$ we solve
\begin{align}\label{eq:Stokespsi}
\begin{aligned}
\Delta \bfu^{\bfpsi}-\nabla\pi^{\bfpsi}=-\Div\bfpsi,\quad\Div\bfu^\bfpsi=0,\quad\text{in}\quad\Omega,\\
\bfu^\bfpsi\cdot\bfn=0,\quad (\nabla\bfu^\bfpsi\,\bfn)_{\bftau}+\alpha (\bfu^\bfpsi)_\bftau=0,\quad\text{on}\quad\partial\Omega,
\end{aligned}
\end{align}
We obtain
\begin{align*}
\int_\Omega\nabla \bfu:\bfPsi\dx=\int_\Omega\nabla \bfu: \nabla \bfu^\Psi\dx+\alpha\int_{\partial\Omega}\bfu\cdot\bfu_{\bfPsi}\,\dd \mathcal H^{n-1}=\tilde f(\bfu^\bfPsi)\leq \norm{f}_{(W^{1,p'}_{n}(\Omega))'}\norm{\nabla \bfu^\bfPsi}_{L^{p'}(\Omega)}.
\end{align*}
In consequence $W^{1,p'}$-estimates imply $W^{1,p}$-estimates.
 
\subsection{Very weak solutions}
In order to get solutions below Sobolev functions we introduce very weak solutions. As data we consider $\tilde{f}$ a distribution on divergence free functions with zero normal trace, that unifies, $f, F$ and $\GG$ and $\tilde{g}$ that unifies $g$ and $h$, i.e., if $g$ and $h$ are measurable functions
\[
\tilde{g}(\psi)=\int_\Omega g\psi\dx-\int_{\partial \Omega}h\, \psi\, \dd\mathcal H^{n-1}\text{ for all }\psi\in C^\infty(\Omega)\text{ with }\partial_n \psi=0\text{ on }\partial \Omega.
\]
 Then we consider the very weak Neumann problem
\[
\skp{w}{\Delta\psi}=\tilde{g}(\psi)\text{ for all }\psi\in C^\infty(\Omega)\text{ with }\partial_n \psi=0\text{ on }\partial \Omega. 
\]
If $w$ was smooth
then
\[
\Delta \psi= g\text{ in }\Omega\text{ with }\partial_n \psi= h\text{ on }\partial \Omega.
\]
Accordingly we then say that $\bfu$ is a very weak solution to \eqref{eq:Stokes}
 if for some given distributions $\tilde f,\tilde g$
\begin{align}
\label{eq:veryweak}
\begin{aligned}
-\skp{\bfu}{\Delta \bfphi} &= \tilde{f}(\phi)\text{ for all }\bfphi\in C^\infty_{\divergence}(\Omega)\text{ s.t. }\bfphi
\cdot \bfn=0\text{ and }(\nabla\bfphi \bfn)_\bftau=-(\alpha \bfphi)_\bftau \text{ on }\partial \Omega,
\\
-\skp{\bfu}{\nabla \psi}&=\tilde{g}(\psi)\text{ for all }\psi\in C^\infty(\Omega)\text{ with }\partial_n \psi=0\text{ on }\partial \Omega.
\end{aligned}
\end{align}
It can be easily checked that if $\bfu$ is smooth enough it is a weak solution. In particular, if for some integrable functions $g$ and $h$
\[
\tilde{g}(\psi)=\int_\Omega g\psi\dx-\int_{\partial \Omega}h\, \psi\, \dd\mathcal H^{n-1},
\]
one finds that $\divergence \bfu=g$ in $\Omega$ and $u\cdot n=h$ on $\partial \Omega$.
One can argue similarly in case $f$ is of the form
\begin{align*}
\tilde{f}(\bfphi)=\int_\Omega \bff\cdot \bfphi+ \bfF\cdot \nabla \phi\dx+\int_{\partial \Omega} \sum_{i=1}^{n-1}\Big((\GG_i-\bfF \bfn\cdot \tau_i)\bfphi\cdot \tau_i  + h \Big)\, \dd\mathcal H^{n-1}
\end{align*}
and find the respective weak formulation.
\subsection{Duality from $W^{2,p'}$ to $L^{p}$ and $W^{2+s,p}$ to $(W^{s,p}_n)'$}
By the dual norm formula we have
\begin{align*}
\norm{\bfu}_{L^p(\Omega)}=\sup_{\bfpsi\in C^\infty_c(\Omega), \norm{\bfpsi}_{L^{p'}(\Omega)}\leq 1} \int_\Omega\bfu\cdot\bfpsi\dx.
\end{align*}
For a fixed $\bfpsi\in C^\infty_c(\Omega)$ with $\norm{\bfpsi}_{L^{p'}}\leq 1$ we solve
\begin{align}\label{eq:Stokespsiweak}
\begin{aligned}
\Delta \bfu^{\bfpsi}-\nabla\pi^{\bfpsi}=\bfpsi,\quad\Div\bfu^\bfpsi=0,\quad\text{in}\quad\Omega,\\
\bfu^\bfpsi\cdot\bfn=0,\quad (\nabla\bfu^\bfpsi\,\bfn)_{\bftau}+\alpha (\bfu^\bfpsi)_\bftau=0,\quad\text{on}\quad\partial\Omega,
\end{aligned}
\end{align}
We obtain
\begin{align*}
&\int_\Omega \bfu\cdot\bfpsi\dx=-\int_\Omega \bfu: (\Delta  \bfu^\Psi-\nabla \pi^{\Psi})\dx=\tilde f(\bfu^\bfPsi) + \tilde{g}(\pi^{\Psi})
\\
&\quad \leq \norm{f}_{(W^{2,p'}_{n}(\Omega))'}\norm{\bfu^\bfPsi}_{W^{2,p'}(\Omega)} + \norm{\tilde{g}}_{(W^{1,p'}_{n}(\Omega))'}\norm{\pi^{\Psi}}_{W^{1,p'}(\Omega)} .
\end{align*}
In consequence $W^{2,p'}$-estimates imply $L^p$-estimates.

For higher order estimates, we can generally introduce estimates of $\bfu$ as a distribution. This means
\begin{align*}
\norm{\bfu}_{(W^{s,p}_n(\Omega))'}:=\sup_{\bfpsi\in C^\infty_{n}(\Omega):\norm{\bfpsi}_{W^{s,p}(\Omega)}\leq 1} \skp{\bfu}{\bfpsi},
\end{align*}
but know
\begin{align*}
\skp{\bfu}{\bfpsi}&=-\skp{\bfu}{\Delta  \bfu^\Psi-\nabla \pi^{\Psi}}=\tilde f(\bfu^\bfPsi)+\tilde{g}(\pi^{\Psi})\\&\leq \norm{f}_{(W^{2+s,p}_{n}(\Omega))'}\norm{\bfu^\bfPsi}_{W^{2+s,p}(\Omega)} + \norm{\tilde{g}}_{(W^{1+s,p}_{n}(\Omega))'}\norm{\pi^{\Psi}}_{W^{1+s,p}(\Omega)}
\end{align*}
\subsection{Duality from $W^{2-\sigma,p'}$ to $W^{\sigma,p}$}
 We take $\sigma<1$ and follow the previous strategy.
By the dual norm formula we find that
\begin{align*}
\norm{\bfu}_{W^{\sigma,p}}\lesssim
\|\nabla\bfu\|_{(W^{1-\sigma,p'}_n(\Omega))'}
=\sup_{\bfpsi\in W^{1-\sigma,p'}_n:\,\|\bfpsi\|_{W^{1-\sigma,p'}}\leq 1}\skp{\nabla\bfu}{\bfpsi}.
\end{align*}
Now we approximate $\tilde{f}$ by smooth functions $\tilde{f}_\epsilon$ (again we can assume that $\tilde{g}=0$ by using the respective argument for the Neumann problem). We solve the system with and obtain the smooth solution $\bfu_\epsilon$. We can follow the steps as in A.1., to find that
\begin{align*}
\skp{\nabla\bfu_\epsilon}{\bfpsi}=\int_\Omega\nabla \bfu_\epsilon: \nabla \bfu^\Psi\dx+\alpha\int_{\partial\Omega}\bfu_\epsilon\cdot\bfu_{\bfPsi}\,\dd \mathcal H^{n-1}=\tilde f_\epsilon(\bfu^\bfPsi)\lesssim \norm{f}_{(W^{1-\sigma,p'}_{n}(\Omega))'}\norm{\bfu^\bfPsi}_{W^{1-\sigma,p'}_n},
\end{align*}
which implies the result by passing with $\epsilon\to 0$.

\section{Neumann problems in irregular domains}
\label{sec:AB}
In this section we consider the Laplace equation with Neumann boundary conditions, i.e.
\begin{align}\label{eq:Laplace}
\Delta u=-\Div\bfF\quad\text{in}\quad\Omega,\quad (\nabla u+\bfF)\cdot\bfn=\chi\quad\text{on}\quad\partial\Omega,
\end{align}
in a domain $\Omega\subset\R^n$ of minimal regularity with normal $\bfn$. The following result gives a counterpart of \cite[Chapter 14]{MaSh} for the Neumann problem. It might be known to experts or at least expected but we were unable to trace a precise reference. 
\begin{theorem}\label{thm:laplace}
Let $n\geq 2$, $p\in(1,\infty)$ and $s\in[1,\infty)$.
Let $\Omega$ be a bounded Lipschitz domain and suppose that either
\begin{enumerate}
\item[(a)]  $p(s-1)\leq n$ and $\partial\Omega$ belongs to the class
$\mathscr M^{s-1/p,p}(\delta)$ for some sufficiently small $\delta>0$ or
\item[(b)] $p(s-1)> n$ and $\partial\Omega$ belongs to the class
$W^{s-1/p,p}$. 
\end{enumerate}
For any $\bfF\in W^{s-1,p}({\Omega})$ and $\chi\in W^{s-1-1/p,p}(\partial\Omega)$
there is a unique solution $u$ to \eqref{eq:Laplace} and we have
\begin{align}\label{eq:estlaplace}
\|u\|_{W^{s,p}({\Omega})}\lesssim\|\bfF\|_{W^{s-1,p}({\Omega})}+\|\chi\|_{W^{s-1-1/p,p}({\partial\Omega})}.
\end{align}
The constant in \eqref{eq:estlaplace} depends on
$\mathscr M^{s-1/p}$- or $W^{s-1/p}$-norms
of the local charts in the parametrisation of  $\partial\Omega$.
\end{theorem}
\begin{remark}\label{rem:general}
The above theorem is actually also true, for any $\tilde{f}\in (W^{s,p}(\Omega))'$. In this case one could just solve the homogeneous Dirichlet problem $-\Delta w= \tilde{f}$ in $\Omega$, $w=0$ on $\partial \Omega$, which transforms it into divergence form.
\end{remark}
\begin{proof}[Proof of Theorem~\ref{thm:laplace}]
As in the proof of Theorems \ref{thm:stokessteadyF} and \ref{thm:stokessteady} we work with local Lipschitz charts $\varphi_1,\dots,\varphi_\ell\in\mathscr M^{s-1/p,p}(\mathbb R^{n-1})(\delta)$ (with extensions $\bfPhi_j$) and corresponding decomposition of unity $(\xi_j)_{j=0}^\ell$ and suppose that the solution is sufficiently smooth. With a notation as there we obtain for $v_j:=(\xi_j u)\circ\bfPhi_j$
\begin{align*}
\Div\big(\bfA_j\nabla v_j)&=-\Div(\bfG_j+\bfH_j)\quad\text{in}\quad\mathbb H,\\
(\bfA_j\nabla v_j)^n&+G^n_j+H^n_j=\chi_j\quad\text{on}\quad\partial\mathbb H,
\end{align*}
where
\begin{align*}
\bfG_j&:=-\nabla\Delta^{-1}_{\mathbb H}([\Delta,\xi_j]u+[\Div,\xi_j]\bfF)\circ\bfPhi_j,\quad\bfH_j:=\mathbf{B}_j(\xi_j\bfF)\circ\bfPhi_j,\\\chi_j&:=(\xi_j\chi)\circ\bfPhi_j+(u\nabla\xi_j)\circ\bfPhi_j\cdot\nabla\varphi_j^\perp,
\end{align*}
and denotes $\Delta^{-1}_{\mathbb H}$ the solution operator to the Laplace equation in $\mathbb H$ with homogenous boundary conditions on $\partial\mathbb H$ (applied to smooth and compactly supported data). 
An equivalent formulation reads as
\begin{align*}
\Delta v_j&=\Div\big((\mathbb I_{n\times n}-\bfA_j)\nabla v_j-\bfG_j-\bfH_j\big)\quad\text{in}\quad\mathbb H,\\
\partial_n v_j&+(\bfA_j-\mathbb I_{n\times n})\nabla v_j)^n+G^n_j+H^n_j=\chi_j\quad\text{on}\quad\partial\mathbb H.
\end{align*}
Estimates for the Laplace equation on the half space are well known and follow for instance by the substraction of a suitable extension and reflection similar to Section~\ref{sec:half} above. Hence we infer that
\begin{align*}
\|v_j\|_{W^{s,p}(\mathbb H)}\lesssim \|(\mathbb I_{n\times n}-\bfA_j)\nabla v_j\big\|_{W^{s-1,p}(\mathbb H)}+\|\bfG_j+\bfH_j\|_{W^{s-1,p}(\mathbb H)}+\|\chi_j\|_{W^{s-1-1/p,p}(\partial\mathbb H)}.
\end{align*} 
It holds by \eqref{lem:9.4.1} and \eqref{eq:MS}
 \begin{align*}
\|&\big(\mathbb I_{n\times n}-\bfA_j)\nabla v_j\big\|_{W^{s-1,p}(\mathbb H)}\\&\lesssim \|\big(\mathbb I_{n\times n}-\bfA_j)\nabla v_j\|_{W^{s-1,p}(\mathbb H)}\\
&\lesssim \|\mathbb I_{n\times n}-\bfA_j\|_{\mathscr M^{s-1,p}_{\tt or}(\mathbb H)}\|\nabla v_j\|_{W^{s-1,p}(\mathbb H)}\\
&\lesssim\|\mathbb I_{n\times n}-\nabla\bfPsi_j^\top\circ\bfPhi_j\|_{\mathscr M^{s-1,p}_{\tt or}(\mathbb H)}\|\nabla v_j\|_{W^{s-1,p}(\mathbb H)}\\&+\|\nabla\bfPsi_j^\top\circ\bfPhi_j(\mathbb I_{n\times n}-\nabla\bfPsi_j\circ\bfPhi_j)\|_{\mathscr M^{s-1,p}_{\tt or}(\mathbb H)}\|\nabla v_j\|_{W^{s-1,p}(\mathbb H)}\\
&\lesssim\big(1+\|\nabla\bfPsi_j^\top\circ\bfPhi_j\|_{\mathscr M^{s-1,p}_{\tt or}(\mathbb H)}\big)\|\mathbb I_{n\times n}-\nabla\bfPsi_j^\top\circ\bfPhi_j\|_{\mathscr M^{s-1,p}_{\tt or}(\mathbb H)}\|\nabla v_j\|_{W^{s-1,p}(\mathbb H)}\\
&\lesssim \|\mathbb I_{n\times n}-\bfA_j)\|_{\mathscr M^{s-1,p}_{\tt or}(\mathbb H)}\|\nabla v_j\|_{W^{s-1,p}(\mathbb H)},
\end{align*}
where
\begin{align*}
\|\nabla\bfPsi_j^\top\circ\bfPhi_j\|_{\mathscr M^{s-1,p}_{\tt or}(\mathbb H)}
&\lesssim \|\bfPsi_j\|_{\mathscr M^{s,p}(\mathbb H)}\lesssim 1+\|\varphi_j\|_{\mathscr M^{s-1/p,p}(\partial\mathbb H)}\lesssim 1,\\
\|(\mathbb I_{n\times n}-\nabla\bfPsi_j\circ\bfPhi_j)\|_{\mathscr M^{s-1,p}_{\tt or}(\mathbb H)}&\lesssim \|\nabla\varphi_j\|_{\mathscr M^{s-1-1/p,p}_{\tt or}(\partial\mathbb H))}\\
&= \|\varphi_j\|_{\mathscr M^{s-1/p,p}(\partial\mathbb H))}\lesssim \delta.
\end{align*}
Note that we used \eqref{eq:MSc} with a suitable choice of the support of of the $\varphi_j$ in the case of $p(s-1)>n$.
 On the other hand, by continuity properties of $\Delta^{-1}_{\mathbb H}$ we have
 \begin{align*}
 \|\bfG_j+\bfH_j\|_{W^{s-1,p} }&\lesssim \|u\circ\bfPhi_j\|_{W^{s-1,p} }+ \|\bfF\circ\bfPhi_j\|_{W^{s-2} }\\
& \lesssim \|u\|_{W^{s-1,p} }+\|\bfF\|_{W^{s-1,p} },
 \end{align*}
 where the hidden constant depends on $\mathrm{det}(\nabla\bfPhi_j)$ and  $\|\bfPhi_j\|_{\mathscr M^{s,p}(\mathbb H)}$ being controlled by \eqref{eq:detJ},
 and \eqref{lem:9.4.1}.

 Let us now suppose that $p\geq 2$. Choosing $s_0\in\R$ such that $W^{1,2}({\mathcal{O}})\hookrightarrow W^{s_0,p}({\mathcal{O}})$, there is $\alpha\in(0,1)$ such that
 \begin{align*}
 \|u\|_{W^{s-1,p} }&\leq \|u\|_{W^{s,p} }^{\alpha}\|u\|_{W^{s_0,p} }^{1-\alpha}\lesssim \| u\|_{W^{s,p} }^{\alpha}\| u\|_{W^{1,2} }^{1-\alpha} \\&\lesssim\| u\|_{W^{s,p} }^{\alpha}\big(\|f\|_{W^{-1,2} }+\|\chi\|_{W^{-1/2,2}}\big)^{1-\alpha}\lesssim\|u\|_{W^{s,p} }^{\alpha}\big(\|f\|_{W^{s-2,p} }+\|\chi\|_{W^{s-1-1/p,p} }\big)^{1-\alpha}\\
&\leq\delta \|u\|_{W^{s,p} }+c(\delta)\big(\|f\|_{W^{s-2,p} }+\|\chi\|_{W^{s-1-1/p,p} }\big).
 \end{align*}
Finally, we clearly have
\begin{align*}
\|\chi_j\|_{W^{s-1-1/p,p}(\partial\mathbb H)}\lesssim \|\chi\|_{W^{s-1-1/p,p}(\partial\Omega)}+ \|u\|_{W^{s-1,p} }
\end{align*}
by \eqref{lem:9.4.1}.
Putting everything together shows for all $j\in\{1,\dots,\ell\}$
 \begin{align*}
\|\nabla v_j\|_{W^{s-1,p}(\mathbb H)}\lesssim \delta
 \|u\|_{W^{s,p}(\Omega)}+\| \bfF\|_{W^{s-1,p}(\Omega)}+\|\chi\|_{W^{s-1-1/p,p}(\partial\Omega)}.
\end{align*}
and can complete the proof as that of Theorem \ref{thm:stokessteady} provided $p\geq 2$.
In particular, we have proved (choosing $s=1$)
\begin{align*}
\|\nabla u\|_{L^{p}({\Omega})}\lesssim\|\bfF\|_{L^{p}({\Omega})}+\|\chi\|_{W^{-1/p,p}({\partial\Omega})}
\end{align*} 
for all $p\geq 2$. Now we can employ a simple duality argument to show that the same estimate holds for $p\in(1,2)$ as well, for that see Appendix~\ref{sec:AA}.
Thus we have for $\delta>0$ arbitrary
\begin{align*}
\|u\|_{W^{s-1,p}(\Omega)}&\leq\,\delta \|u\|_{W^{s,p}(\Omega)}+c(\delta)\|\nabla u\|_{L^{p}(\Omega)}\\
&\leq\,\delta \|u\|_{W^{s,p}(\Omega)}+c(\delta) \big(\|f\|_{W^{-1,p}(\Omega)}+\|\chi\|_{W^{-1/p,p}({\partial\Omega})}\big)\\
&\leq\,\delta\|u\|_{W^{s,p}(\Omega)}+c(\delta)\big( \|f\|_{W^{s-2,p}(\Omega)}+\|\chi\|_{W^{s-1-1/p,p}({\partial\Omega})}\big)
\end{align*}
and the proof can be completed as before.
\end{proof}
As in Section \ref{sec:stokessteady} we can argue similarly for the problem in non-divergence form
\begin{align}\label{eq:Laplacef}
-\Delta u=f\quad\text{in}\quad\Omega,\quad \nabla u\cdot\bfn=\chi\quad\text{on}\quad\partial\Omega,
\end{align}
satisfying the compatibility condition $\int_{\Omega} f\dx=\int_{\partial\Omega} \chi\, \dd \mathcal{H}^{n-1}$.
Accordingly obtain on the right-hand side of the transformed problem
$$g_j:=\mathrm{det}(\nabla\bfPhi_j)([\Delta,\xi_j]u-f\xi_j)\circ\bfPhi_j,\quad \chi_j=(\xi_j\chi)\circ\bfPhi_j.$$
We thus obtain the following result.
\begin{theorem}\label{thm:laplaceF}
Let $n\geq 2$, $p\in(1,\infty)$ and $s\in[2,\infty)$.
Let $\Omega$ be a bounded Lipschitz domain and suppose that either
\begin{enumerate}
\item[(a)]  $p(s-1)\leq n$ and $\partial\Omega$ belongs to the class
$\mathscr M^{s-1/p,p}(\delta)$ for some sufficiently small $\delta>0$ or
\item[(b)] $p(s-1)> n$ and $\partial\Omega$ belongs to the class
$W^{s-1/p,p}$. 
\end{enumerate}
For any $f\in W^{s-2,p}({\Omega})$ and $\chi\in W^{s-1-1/p,p}(\partial\Omega)$ satisfying $\int_{\Omega} f\dx=\int_{\partial\Omega} \chi\,\dd \mathcal{H}^{n-1}$
there is a unique solution $u$ to \eqref{eq:Laplace} and we have
\begin{align}\label{eq:estlaplaceF}
\|u\|_{W^{s,p}({\Omega})}\lesssim\|f\|_{W^{s-2,p}({\Omega})}+\|\chi\|_{W^{s-1-1/p,p}({\partial\Omega})}.
\end{align}
The constant in \eqref{eq:estlaplaceF} depends on
$\mathscr M^{s-1/p}$- or $W^{s-1/p}$-norms
of the local charts in the parametrisation of  $\partial\Omega$.
\end{theorem}

We consider now the very weak Neumann problem
\[
\skp{w}{\Delta\psi}=\tilde{g}(\psi)\text{ for all }\psi\in C^\infty(\Omega)\text{ with }\partial_n \psi=0\text{ on }\partial \Omega
\]
for some distribution $\tilde g$.
If $w$ was smooth and 
\[
\tilde{g}(\psi)=\int_\Omega g\psi\dx-\int_{\partial \Omega}h \psi\dx
\]
for some given functions $g$ and $h$,
then
\[
\Delta \psi= g\text{ in }\Omega\text{ with }\partial_n \psi= h\text{ on }\partial \Omega.
\]
As in Appendix \ref{sec:AA} one can now use duality
to include non-homogeneous boundary data by considering the trace space
\begin{align*}
\mathscr W^{s-1-1/p,p}(\partial\Omega)&:=\{\nabla v\cdot\bfn|_{\partial\Omega}:\,v\,\in W^{s,p}(\Omega)\},\\
\|\chi\|_{\mathscr W^{s-1-1/p,p}(\partial\Omega)}&:=\inf\{\|v\|_{W^{s,p}(\Omega)}:\,\,,\,v\in W^{s,p}(\Omega),\,\,\nabla v\cdot\bfn=\chi\},
\end{align*}
where we understand the trace $\nabla v\cdot\bfn|_{\partial\Omega}$ and the equality  in the sense of distributions.
For $s>1+1/p$ one clearly has that $\mathscr W^{s-1-1/p,p}(\partial\Omega)=W^{s-1-1/p,p}(\partial\Omega)$, but this relationship is lost if $s\leq 1+1/p$.
In fact, it is more natural for very-weak spaces to consider $\chi$ and $f$ together as a general functional $\tilde{f}\in (W^{s,p}(\Omega))'$. For these one can introduce the very weak solution as a distribution. Namely that
\begin{align}\label{eq:Stokesdivneg}
-\skp{u}{\Delta\phi}=\skp{\tilde{f}}{\phi},
\end{align}
for all $\phi\in C^\infty(\Omega)$ with $\partial_n \phi=0$.
Indeed, in case there exist (smooth) functions $f$ and $\chi$, for which
\[
\skp{\tilde{f}}{\phi}=\int_\Omega \Big(f \phi-\mathbf{F}\cdot \nabla \phi\Big)\dx +\int_{\partial \Omega} \chi\phi\, \dd\mathcal H^{n-1}
\]
any smooth very weak solution satisfies (by density) that
\[
-\int_{\Omega} \Delta u \phi\dx+\int_{\partial\Omega} \partial_n u \phi\dx= \int_\Omega (f+\divergence\mathbf{F})\phi\dx +\int_{\partial \Omega} (\chi-\mathbf{F}\cdot {\bf n}) \phi\, \dd \mathcal H^{n-1},
\]
for all $\phi\in C^\infty(\Omega)$.
But this implies the strong partial differential equation with respective boundary values. Hence we may formulate the following corollary with no loss of generality:
\begin{corollary}\label{cor:laplace}
Let $p\in(1,\infty)$ and $\sigma<1$.
Let $\Omega$ be a bounded Lipschitz domain and suppose for $s:=2-\sigma$
that either
\begin{enumerate}
\item[(a)]  $p(s-1)\leq n$ and $\partial\Omega$ belongs to the class
$\mathscr M^{s+1-1/p,p}(\delta)$ for some sufficiently small $\delta>0$ or
\item[(b)] $p(s-1)> n$ and $\partial\Omega$ belongs to the class
$W^{s+1-1/p,p}$. 
\end{enumerate}
For any $\tilde{f}\in (W^{2-\sigma,p'}({\Omega}))'$ 
there is a unique solution to \eqref{eq:Stokesdivneg} and we have
\begin{align}\label{eq:laplacecor}
\|u\|_{W^{\sigma,p}(\Omega)}
&\lesssim\|f\|_{(W^{2-\sigma,p'}(\Omega))'}
\end{align}
\end{corollary}
%

\smallskip
\par\noindent
{\bf Acknowledgement}. 
D. Breit has been funded by Grant BR 4302/3-1 (525608987) of the German Research Foundation (DFG) within the framework of the priority research program SPP 2410 and by Grant BR 4302/5-1 (543675748) of the German Research Foundation (DFG).

S. Schwarzacher is partially supported by the ERC-CZ Grant CONTACT LL2105 funded by the Ministry of Education, Youth and Sport of the Czech Republic and by the Charles University Research Centre program No. UNCE/24/SCI/005. S. Schwarzacher also acknowledges the support of the VR
Grant 2022-03862 of the Swedish Research Council and is a member of the Ne\v{c}as Centre for Mathematical Modeling.

\section*{Compliance with Ethical Standards}\label{conflicts}
\smallskip
\par\noindent
{\bf Conflict of Interest}. The author declares that he has no conflict of interest.

\smallskip
\par\noindent
{\bf Data Availability}. Data sharing is not applicable to this article as no datasets were generated or analysed during the current study.


\begin{thebibliography}{[M]}
\bibitem{AACG} P. Acevedo Tapia, C. Amrouche, C. Concac, A. Ghosh: Stokes and Navier-Stokes equations with Navier boundary conditions. J. Diff. Equ. 285, 258--320. (2021)
\bibitem{AR} C. Amrouche, A. Rejaiba: $L^p$-theory for Stokes and Navier--Stokes equations with Navier boundary condition, J. Differ. Equ. 256, 1515--1547. (2014)
\bibitem{Br} D. Breit, Regularity results in 2D fluid-structure interaction. {\em Math. Ann.}
388, 1495--1538. (2024)
\bibitem{Br2} D. Breit: A Schauder theory for the Stokes equations in rough domains. To appear in \emph{Indiana Univ. Math. J.} Preprint at www.iumj.indiana.edu/IUMJ/Preprints/60363.pdf
\bibitem{BulBurSch16}
M.~Bul\'\i\v{c}ek, J.~Burczak, and S.~Schwarzacher.
\newblock A unified theory for some non-{N}ewtonian fluids under singular
  forcing.
\newblock {\em SIAM J. Math. Anal.}, 48:4241--4267, 2016.
\bibitem{BMR09} Bulicek, M., Malek, J.,  Rajagopal, K. R. (2009). Mathematical analysis of unsteady flows of fluids with pressure, shear-rate, and temperature dependent material moduli that slip at solid boundaries. SIAM journal on mathematical analysis, 41(2), 665-707.
\bibitem{Ca} L. Cattabriga: Su un problema al contorno relativo al sistema di equazioni di Stokes. Rend. Semin. Mat. Univ. Padova 31, 308--340. (1961)
\bibitem{CiMa} A. Cianchi, V. G. Maz'ya: Second-Order Two-Sided Estimates in
Nonlinear Elliptic Problems. Arch Rational Mech. Anal. 229 569--599. (2018)
\bibitem{Ga} G. P. Galdi: An Introduction to the Mathemaical Theory of the Navier--Stokes equations. Steady-Sate Problems. 2nd Edition. Springer Monographs in Mathematics. Springer, New York Dordrecht Heidelberg London. (2011) 
\bibitem{JM00} Jager, W.,  Mikelic, A. (2001). On the roughness-induced effective boundary conditions for an incompressible viscous flow. Journal of Differential Equations, 170(1), 96-122.
\bibitem{KSW1} M. K\"ohne, J. Saal, and L. Westermann: Optimal Sobolev Regularity for the Stokes Equations on a
2D Wedge Domain. Math. Ann. 379, 377--413. (2021)
\bibitem{KSW} M. K\"ohne, J. Saal \& L. Westermann: Optimal Regularity for the Stokes Equations on a 2D Wedge Domain Subject to Navier Boundary Conditions. Preprint at arXiv:2410.24063.
\bibitem{LSU} O.~A.~Ladyzhenskaia, V.~A.~Solonnikov, and N.~N.~Ural'tseva: {\em Linear and quasi-linear equations of parabolic type.} Vol. 23. American Mathematical Soc., 1968.
\bibitem{MS} V. Macha, S. Schwarzacher: \emph{Global continuity and BMO estimates for non-Newtonian fluids with
perfect slip boundary conditions}, Rev. Mat. Iberoamericana, 1115--1173. (2021)
\bibitem{MaSh} V. G. Maz'ya, T. O. Shaposhnikova. Theory of Sobolev multipliers, volume 337 of
Grundlehren der Mathematischen Wissenschaften [Fundamental Principles of Mathematical Sciences].
Springer-Verlag, Berlin. With applications to differential and integral operators. (2009)
\bibitem{Na} C.L.M.H. Navier, M\'emoire sur les lois du mouvement des fluides, M\'em. Acad. Sci. Inst. Fr. (2), 389--440. (1823)
\bibitem{RuSi} T. Runst, W. Sickel, Sobolev Spaces of Fractional Order, Nemytskij Operators, and Nonlinear Partial Differential Equations, De Gruyter Series in Nonlinear Analysis and Applications, vol.3, Walter de Gruyter \& Co., Berlin, New York. (1996)
\bibitem{So} H. Sohr, \emph{The Navier-Stokes equations. An elementary functional analytic approach.}
Birkh\"auser Advanced Texts, Birkh\"auser Verlag, Basel. (2001)
\bibitem{Te} R. Temam, \emph{Navier-Stokes Equations.} North-Holland Pub. Co. Amsterdam-New York-
Tokyo. (1977)

\bibitem{Tr} H. Triebel, \emph{Theory of Function Spaces}, Modern Birkh\"auser Classics, Springer, Basel. (1983)

\bibitem{Tr2} H. Triebel, \emph{Theory of Function Spaces II}, Modern Birkh\"auser Classics, Springer, Basel. (1992)
%
\end{thebibliography}
\end{document}